\documentclass[ssy,preprint]{imsart}%
\usepackage{amsfonts,amssymb,amsthm,amsmath,natbib}
\usepackage{enumerate,color}
\usepackage{subfig}
\usepackage{caption}
\usepackage{graphicx}
\RequirePackage[colorlinks,citecolor=blue,urlcolor=blue]{hyperref}
\arxiv{1104.0347}
\startlocaldefs

\DeclareMathOperator{\diag}{diag}
\DeclareMathOperator*{\argmin}{arg\,min}
\numberwithin{equation}{section}
\theoremstyle{plain}

\newtheorem{proposition}{Proposition}
\newtheorem{theorem}{Theorem}

\newtheorem{conjecture}[theorem]{Conjecture}

\def\Cplusplus{{\rm C\raise.3ex\hbox{\small++ }}}
\endlocaldefs
\begin{document}
\begin{frontmatter}
\title{Many-server queues with customer abandonment: Numerical analysis of their
diffusion models\protect\thanksref{T1}}
\runtitle{Numerical analysis of diffusion models}
\thankstext{T1}{Supported in part by NSF Grants CMMI-0727400, CMMI-0825840,
and CMMI-1030589.}
\begin{aug}
\author{\fnms{Shuangchi} \snm{He}
\ead[label=e2]{heshuangchi@gatech.edu}}
\and
\author{\fnms{ J. G.} \snm{Dai}
\ead[label=e1]{dai@gatech.edu}}
\runauthor{S. He and J. G. Dai}
\affiliation{Georgia Institute of Technology}
\address{Shuangchi He\\
H. Milton Stewart School of Industrial\\
~~ and Systems Engineering\\
Georgia Institute of Technology\\
Atlanta, Georgia 30332, USA\\
\printead{e2}}
\address{J. G. Dai\\
H. Milton Stewart School of Industrial\\
~~ and Systems Engineering\\
Georgia Institute of Technology\\
Atlanta, Georgia 30332, USA\\
\printead{e1}}
\end{aug}
\begin{abstract}
We use multidimensional diffusion processes to approximate the dynamics of a
queue served by many parallel servers. The queue is served in the
first-in-first-out (FIFO) order and the customers waiting in queue may abandon
the system without service. Two diffusion models are proposed in this paper.
They differ in how the patience time distribution is built into them. The
first diffusion model uses the patience time density at zero and the second
one uses the entire patience time distribution. To analyze these diffusion
models, we develop a numerical algorithm for computing the stationary
distribution of such a diffusion process. A crucial part of the algorithm is
to choose an appropriate reference density. Using a conjecture on the tail
behavior of a limit queue length process, we propose a systematic approach to
constructing a reference density. With the proposed reference density, the
algorithm is shown to converge quickly in numerical experiments. These
experiments also show that the diffusion models are good approximations for
many-server queues, sometimes for queues with as few as twenty servers.
\end{abstract}
\end{frontmatter}

\section{Introduction}

\label{sec:Introduction}

The focus of this paper is the numerical analysis of multidimensional
diffusion processes that approximate the dynamics of a queue with many
parallel servers. A many-server queue serves as a building block modeling
operations of a large-scale service system. Such a service system may be a
call center with hundreds of agents, a hospital department with tens or
hundreds of inpatient beds, or a computer cluster with many processors. When
the customers of a service system are human beings, some of them may abandon
the system before their service begins. The phenomenon of customer abandonment
is ubiquitous because no one would wait for service indefinitely. As argued in
\cite{GMR02}, one must model customer abandonment explicitly in order for an
operational model to be relevant for decision making. We model customer
abandonment by assigning each customer a patience time. When a customer's
waiting time for service exceeds his patience time, he abandons the queue
without service.

The exact analysis of such a many-server queue has been largely limited to an
$M/M/n+M$ model (also called an Erlang-A model) that has a Poisson arrival
process and exponential service and patience time distributions. See, e.g.,
\cite{GMR02}. However, as pointed out by \cite{BMSZZS05}, the service time
distribution in a call center appears to follow a log-normal distribution.
Such distributions have also been observed by \cite{acdds10} for lengths of
stay in a hospital. Moreover, the patience time distribution in a call center
has been observed to be far from exponential by \cite{ZelMan05}. With a
general service or patience time distribution, there is no finite-dimensional
Markovian representation of the queue. Except computer simulations, there is
no method to exactly analyze such a queue either analytically or numerically.
To deal with the challenge, the following strategies are adopted in this paper
for analyzing a many-server queue.

First, the service time distribution is restricted to be phase-type. Since
phase-type distributions can be used to approximate any positive-valued
distribution, such a queueing model is still relevant to practical systems. We
focus on a $GI/Ph/n+GI$ queue with $n$ identical servers. The first $GI$
indicates that the customer interarrival times are independent and identically
distributed (iid) following a general distribution, the $Ph$ indicates that
the service times are iid following a phase-type distribution, and the $+GI$
indicates that the patience times are iid following a general distribution.
Second, we are particularly interested in a queue operating in the
\emph{Quality- and Efficiency-Driven (QED)} regime: The queue has a large
number of servers and the arrival rate is high; the arrival rate and the
service capacity are approximately balanced so that the mean waiting time is
relatively short compared with the mean service time. As argued in
\cite{GMR02}, such a system has high server utilization as well as short
customer waiting times and a small fraction of abandonment. Therefore, both
quality and efficiency can be achieved in this regime. Third, rather than
analyzing the many-server queue itself, we propose and analyze diffusion
models that approximate the queue. Two diffusion models are proposed in this
paper. In each model, a multidimensional diffusion process is used
to represent the scaled customer numbers among service phases. The difference
between the two diffusion models lies in how the patience time distribution is
built into them. The first diffusion model uses the patience time density at
zero and the second one uses the entire patience time distribution. In
particular, the diffusion process in the first model is a multidimensional
piecewise Ornstein-Uhlenbeck (OU) process. We propose an algorithm in this
paper to numerically solve the stationary distribution of a diffusion process.
The computed stationary distribution is used to estimate the performance
measures of a many-server queue. Numerical examples in
Section~\ref{sec:Numerical} demonstrate that the diffusion models are very
accurate in predicting the performance of a many-server queue, even if the
queue has as few as twenty servers.

Except for the one-dimensional case, the stationary distribution of a
piecewise OU process has no explicit formula. The algorithm proposed in this
paper is a variant of the one in \cite{daihar92}, which computes the
stationary distribution of a semimartingale reflecting Brownian motion (SRBM).
As in \cite{daihar92}, the starting point of our algorithm is the basic
adjoint relationship that characterizes the stationary distribution of a
diffusion process. With an appropriate reference density, the algorithm can
produce a stationary density that satisfies this relationship.

We set up a Hilbert space using the reference density. In this space, the
stationary density is orthogonal to an infinite-dimensional subspace $H$. A
finite-dimensional subspace $H_{k}$\ is used to approximate $H$\ and a
function orthogonal to $H_{k}$ can be numerically computed by solving a system
of finitely many linear equations. This function is used to approximate the
stationary density. There are two sources of error in computing the
approximate stationary density by our algorithm: \emph{approximation error}
and \emph{round-off error}. The approximation error arises because $H_{k}$ is
an approximation of $H$. As $H_{k}$ increases to $H$, the approximation error
decreases to zero. The round-off error occurs because the solution to the
system of linear equations has error due to the finite precision of a
computer. As $H_{k}$ increases to $H$, the dimension of the linear system gets
higher and the coefficient matrix becomes closer to singular. As a
consequence, the round-off error increases. The condition number of the matrix
is used as a proxy for the round-off error. Balancing the approximation error
and the round-off error is an important issue in our algorithm.

A properly chosen reference density is essential for the convergence of the
algorithm. By convergence, we mean that the approximation error converges to
zero as $H_{k}$ increases to $H$. More importantly, a \textquotedblleft
good\textquotedblright\ reference density can make $H_{k}$ converge to $H$
quickly so that the resulting approximation error and round-off error are
small simultaneously even though the dimension of $H_{k}$ is moderate. To
ensure the convergence of the algorithm, the reference density should have a
comparable or slower decay rate than the stationary density. Since the
stationary density is unknown, we make a conjecture on the tail behavior of
the limit queue length process of many-server queues with customer
abandonment. We conjecture that the limit queue length process has a Gaussian
tail and the tail depends on the service time distribution only through its
first two moments. This tail is used to construct a product-form reference
density. With this reference density, the algorithm appears to converge
quickly, producing stable and accurate results. For comparison purposes, we
also test the algorithm with a certain \textquotedblleft
naively\textquotedblright\ chosen reference density in
Section~\ref{sec:InfluenceReference}. The algorithm fails to converge with the
\textquotedblleft naive\textquotedblright\ reference density. The major
contributions of this paper are the proposed diffusion models and the proposed
reference densities that are critical to the numerical algorithm for computing
the stationary distribution of a diffusion model.

Our diffusion models are obtained by replacing certain scaled renewal
processes by Brownian motions. The replacement procedure is rooted in the
many-server heavy traffic limit theorems that are proved in an asymptotic
regime. The two diffusion model proposed in this paper are motivated by the
diffusion limits proved in \cite{DaiHeTezcan10} and \cite{ReedTezcan09}. See
Section~\ref{sec:models} for more details. The theory of diffusion
approximation for many-server queues can be traced back to the seminal paper
by \cite{HalWhi81}, where a diffusion limit was established for $GI/M/n$
queues. \cite{GMR02} proved a diffusion limit for $M/M/n+M$ queues that allows
for customer abandonment, and \cite{Whitt05a} generalized the result to
$G/M/n+M$ queues. \cite{PuhalskiiReiman00} established a diffusion limit for
$GI/Ph/n$ queues. Their result was extended to $G/Ph/n+GI$ queues with
customer abandonment in \cite{DaiHeTezcan10}. Recently, \cite{ReedTezcan09}
proved a diffusion limit for $GI/M/n+GI$ queues. In their framework, a refined
limit process is obtained by scaling the patience time hazard rate function.

\cite{harngu90} derived Brownian models for multiclass open queueing networks.
Their diffusion models are SRBMs and are rooted in the conventional heavy
traffic limit theorems that are pioneered in \cite{iglwhi70b} for serial
networks and \cite{rei84} for single-class networks. See \cite{Williams96} for
a survey of limit theorems in literature. For a two-dimensional SRBM living in
a rectangle, \cite{daihar91} proposed an algorithm computing its stationary
distribution. \cite{daihar92} extended the algorithm for an SRBM living in an
orthant. To deal with the unbounded state space, the notion of a reference
density was first introduced there. Their finite-dimensional space $H_{k}$ is
constructed via (global) multinominals of order at most $k$. With this choice
of $H_{k}$, the algorithm appears numerically unstable occasionally. In such a
case, the round-off error may dominate the approximation error while the
approximation error is still significant. \cite{scdd02} extended
\cite{daihar91} to a hypercube state space of an arbitrary dimension. They
used a finite element method to construct $H_{k}$ to avoid numerical
instability. Their algorithm sometimes converges slowly because they did not
explore a reference density. A linear programming algorithm for computing the
stationary distribution of a diffusion process was proposed in
\cite{SaureGlynnZeevi10}. Both SRBMs in an orthant and a diffusion
approximation of many-server queues with two priority classes were
investigated in their paper. Like the role of the reference density, it
appears that the rescaling of variables is essential to the convergence of
their algorithm.

The remainder of the paper is organized as follows. General diffusion
processes are introduced in Section~\ref{sec:DiffusionProcesses}, where the
basic adjoint relationship for a diffusion process is also presented. In
Section~\ref{sec:Algorithm}, we begin with recapitulating the generic
algorithm of \cite{daihar92}, and then propose a finite element implementation
that follows \cite{scdd02}. Two diffusion models for $GI/Ph/n+GI$ queues are
presented in Section~\ref{sec:DiffusionQueues}. In
Section~\ref{sec:ReferenceDensity}, we discuss how to choose an appropriate
reference density exploiting the tail behavior of a diffusion process. In
Section~\ref{sec:Numerical}, it is demonstrated via numerical examples that
the diffusion models serve as good approximations of many-server queues.
Section~\ref{sec:Implementation} is dedicated to some implementation issues
arising from the proposed algorithm. The paper is concluded in
Section~\ref{sec:Conclusion}. We leave the proofs of
Propositions~\ref{prop:independence} and~\ref{prop:Hn} to the appendix.

\subsection*{Notation}

The symbols $\mathbb{N}$, $\mathbb{R}$, and $\mathbb{R}_{+}$ are used to
denote the sets of positive integers, real numbers, and nonnegative real
numbers, respectively. For $d,m\in\mathbb{N}$, $\mathbb{R}^{d}$ denotes the
$d$-dimensional Euclidean space and $\mathbb{R}^{d\times m}$ denotes the space
of $d\times m$ real matrices. We use $C_{b}^{2}(\mathbb{R}^{d})$ to denote the
set of real-valued functions on $\mathbb{R}^{d}$ that are twice continuously
differentiable with bounded first and second derivatives. For $z,w\in
\mathbb{R}$, we set $z^{+}=\max\{z,0\}$, $z^{-}=\max\{-z,0\}$, and $z\wedge
w=\min\{z,w\}$. All vectors are envisioned as column vectors. For a
$d$-dimensional vector $x\in\mathbb{R}^{d}$, we use $x_{j}$ for its $j$th
entry and $\operatorname*{diag}(x)$ for the $d\times d$ diagonal matrix with
$j$th diagonal entry $x_{j}$. For a matrix $M$, $M^{\prime}$ denotes its
transpose, $M_{ij}$ denotes its $(i,j)$th entry, and $|M|=(\sum_{i,j}%
M_{ij}^{2})^{1/2}$. We reserve $I$ for the $d\times d$ identity matrix, $e$
for the $d$-dimensional vector of ones, and $e^{j}$ for the $d$-dimensional
vector with its $j$th entry one and all other entries zero. Given two
functions $\varphi$ and $\hat{\varphi}$ from $\mathbb{N}$ to $\mathbb{R}$, we
write $\hat{\varphi}(n)=O(\varphi(n))$ as $n\rightarrow\infty$ if there exists
a constant $\kappa>0$ and some $n_{0}\in\mathbb{N}$ such that $|\hat{\varphi
}(n)|\leq\kappa|\varphi(n)|$ for all $n>n_{0}$.

\section{Diffusion processes}

\label{sec:DiffusionProcesses}

Let $d$ be a positive integer. This paper focuses on a $d$-dimensional
diffusion process $X=\{X(t):t\geq0\}$. Let $(\Omega,\mathcal{F},\mathbb{F}%
,\mathbb{P})$\ be a filtered probability space with filtration $\mathbb{F}%
=\{\mathcal{F}_{t}:t\geq0\}$. We assume that $X$ satisfies the following
stochastic differential equation
\begin{equation}
X(t)=X(0)+\int_{0}^{t}b(X(s))\,\mathrm{d}s+\int_{0}^{t}\sigma
(X(s))\,\mathrm{d}B(s), \label{eq:diffusion}%
\end{equation}
where the drift coefficient $b$ is a function from $\mathbb{R}^{d}$ to
$\mathbb{R}^{d}$, the diffusion coefficient $\sigma$ is a function from
$\mathbb{R}^{d}$ to $\mathbb{R}^{d\times m}$, and $B=\{B(t):t\geq0\}$ is an
$m$-dimensional standard Brownian motion with respect to $\mathbb{F}$. We
assume that both $b$ and $\sigma$ are Lipschitz continuous, i.e., there exists
a constant $c_{1}>0$ such that
\begin{equation}
\left\vert b(x)-b(y)\right\vert +\left\vert \sigma(x)-\sigma(y)\right\vert
\leq c_{1}|x-y|\qquad\text{for all }x,y\in\mathbb{R}^{d}\text{.}
\label{eq:Lipschitz}%
\end{equation}
Under condition (\ref{eq:Lipschitz}), the stochastic differential equation
(\ref{eq:diffusion}) has a unique strong solution, i.e., there exists a unique
process $X$ on $(\Omega,\mathcal{F},\mathbb{F},\mathbb{P})$ such that (a) $X$
is adapted to $\mathbb{F}$, (b) for each sample path $\omega\in\Omega$,
$X(t,\omega)$ is continuous in $t$, and (c) for each $t\geq0$, the stochastic
differential equation (\ref{eq:diffusion}) holds with probability one. See
\cite{Oksendal03} for more details. We also assume that $\sigma$ is uniformly
elliptic, i.e., there exists a constant $c_{2}>0$ such that
\begin{equation}
y^{\prime}\Sigma(x)y\geq c_{2}y^{\prime}y\qquad\text{ for all }x,y\in
\mathbb{R}^{d}\text{,} \label{eq:uelliptic}%
\end{equation}
where
\begin{equation}
\Sigma(x)=\sigma(x)\sigma^{\prime}(x). \label{eq:Sigma}%
\end{equation}

We are interested in the diffusion processes that model the dynamics of a
queue with many parallel servers. Parallel-server queues will be introduced in
Section~\ref{sec:DiffusionQueues}. In that section, two diffusion processes
will be identified to model such a queue and the coefficients $b$ and $\sigma$
will be mapped out explicitly in terms of primitive data of the queue. The
diffusion models presented in Section~\ref{sec:DiffusionQueues} are rooted in
many-server heavy traffic limit theorems proved in \cite{DaiHeTezcan10} and
\cite{ReedTezcan09}.

A probability distribution $\pi$ on $\mathbb{R}^{d}$ is said to be a
stationary distribution of $X$ if $X(t)$ follows distribution $\pi$ for each
$t>0$ whenever $X(0)$ has distribution $\pi$. Condition (\ref{eq:uelliptic})
is required to ensure the uniqueness of the stationary distribution. See
\cite{DiekerGao11} for more details. In this paper, we assume that $X$ has a
unique stationary distribution $\pi$ and $\pi$ has a density $g$ with respect
to the Lebesgue measure on $\mathbb{R}^{d}$. For a general diffusion process,
there is no explicit solution for $\pi$. This paper develops a numerical
algorithm computing $\pi$. As in \cite{daihar92}, the starting point of the
algorithm is the basic adjoint relationship
\begin{equation}
\int_{\mathbb{R}^{d}}\mathcal{G}f(x)\,\pi(\mathrm{d}x)=0\qquad\text{for all
}f\in C_{b}^{2}(\mathbb{R}^{d}), \label{eq:bar}%
\end{equation}
where $\mathcal{G}$ is the generator of $X$ defined by
\begin{equation}
\mathcal{G}f(x)=\sum_{j=1}^{d}b_{j}(x)\frac{\partial f(x)}{\partial x_{j}%
}+\frac{1}{2}\sum_{j=1}^{d}\sum_{\ell=1}^{d}\Sigma_{j\ell}(x)\frac
{\partial^{2}f(x)}{\partial x_{j}\partial x_{\ell}}\qquad\text{for each }f\in
C_{b}^{2}(\mathbb{R}^{d}) \label{eq:generator}%
\end{equation}
and $\Sigma$ is the covariance matrix given by (\ref{eq:Sigma}). The following
theorem is a consequence of Proposition~9.2 in \cite{ethkur86}.

\begin{theorem}
\label{thm:converse} Let $\pi$ be a probability distribution on $\mathbb{R}%
^{d}$ that satisfies (\ref{eq:bar}). Then, $\pi$ is a stationary distribution
of $X$.
\end{theorem}

In this paper, we conjecture that a stronger version of Theorem
\ref{thm:converse} is true.

\begin{conjecture}
\label{conj:signedMeasure} Let $\pi$ be a signed measure on $\mathbb{R}^{d}$
that satisfies (\ref{eq:bar}) and $\pi(\mathbb{R}^{d})=1$. Then, $\pi$ is a
nonnegative measure and consequently it is a stationary distribution of $X$.
\end{conjecture}

Our algorithm is to construct a function $g$ on $\mathbb{R}^{d}$ such that%
\begin{equation}
\int_{\mathbb{R}^{d}}g(x)\,\mathrm{d}x=1\qquad\text{and}\qquad\int%
_{\mathbb{R}^{d}}\mathcal{G}f(x)g(x)\,\mathrm{d}x=0\qquad\text{for all }f\in
C_{b}^{2}(\mathbb{R}^{d})\text{.} \label{eq:bar_density}%
\end{equation}
Assuming that Conjecture \ref{conj:signedMeasure} is true, $g$ must be the
unique stationary density of $X$. As a special case, the nonnegativity of a
signed measure $\pi$ that satisfies (\ref{eq:bar}) for a piecewise OU process
was proposed as an open problem by \cite{DaiDieker10}. Piecewise OU processes
will be introduced in Section~\ref{sec:PatienceZero}.

\section{A finite element algorithm for stationary distributions}

\label{sec:Algorithm}

In this section, we propose a numerical algorithm computing the stationary
density $g$. The basic algorithm follows the one developed in \cite{daihar92}.
The finite element implementation closely follows \cite{scdd02}.

\subsection{A reference density}

\label{sec:ReferenceRatio}To compute the stationary density $g$, we adopt a
notion called a \emph{reference density} that was first introduced by
\cite{daihar92}. A reference density for $g$ is a function $r$ defined from
$\mathbb{R}^{d}$ to $\mathbb{R}_{+}$ such that
\begin{equation}
\int_{\mathbb{R}^{d}}r(x)\,\mathrm{d}x<\infty\qquad\text{and\qquad}%
\int_{\mathbb{R}^{d}}q^{2}(x)r(x)\,\mathrm{d}x<\infty\text{,} \label{eq:ref}%
\end{equation}
where
\[
q(x)=\frac{g(x)}{r(x)}\qquad\text{for each }x\in\mathbb{R}^{d}%
\]
is called the \emph{ratio function}. Such a function $r$ exists because $r=g$
is a reference density. The reference density controls the convergence of our
algorithm. We will discuss how to choose a reference density for the diffusion
models of a many-server queue in Section~\ref{sec:ReferenceDensity}.

For the rest of Section~\ref{sec:Algorithm}, we assume that a reference
density $r$ satisfying (\ref{eq:ref}) has been determined and remains fixed.
In addition, we assume that%
\begin{equation}
\int_{\mathbb{R}^{d}}b_{j}^{2}(x)r(x)\,\mathrm{d}x<\infty\qquad\text{and}%
\qquad\int_{\mathbb{R}^{d}}\Sigma_{j\ell}^{2}(x)r(x)\,\mathrm{d}%
x<\infty\label{eq:Gassump}%
\end{equation}
for $j,\ell=1,\ldots,d$. Since both $b$ and $\sigma$ are Lipschitz continuous,
condition (\ref{eq:Gassump}) is satisfied if
\begin{equation}
\int_{\mathbb{R}^{d}}\left\vert x\right\vert ^{4}r(x)\,\mathrm{d}x<\infty.
\label{eq:refcond}%
\end{equation}
Let $L^{2}(\mathbb{R}^{d},r)$ be the space of all square-integrable functions
on $\mathbb{R}^{d}$ with respect to the measure that has density $r$, i.e.,%
\[
L^{2}(\mathbb{R}^{d},r)=\Big\{f\in\mathcal{B}(\mathbb{R}^{d}):\int%
_{\mathbb{R}^{d}}f^{2}(x)r(x)\,\mathrm{d}x<\infty\Big\}
\]
where $\mathcal{B}(\mathbb{R}^{d})$ is the set of Borel-measurable functions
on $\mathbb{R}^{d}$. Condition (\ref{eq:ref}) implies that $q\in
L^{2}(\mathbb{R}^{d},r)$. We define an inner product on $L^{2}(\mathbb{R}%
^{d},r)$ by%
\[
\langle f,\hat{f}\rangle=\int_{\mathbb{R}^{d}}f(x)\hat{f}(x)r(x)\,\mathrm{d}%
x\qquad\text{for }f,\hat{f}\in L^{2}(\mathbb{R}^{d},r)\text{.}%
\]
The induced norm is given by%
\begin{equation}
\left\Vert f\right\Vert =\langle f,f\rangle^{1/2}\qquad\text{for each }f\in
L^{2}(\mathbb{R}^{d},r)\text{.} \label{eq:norm}%
\end{equation}
One can check that $L^{2}(\mathbb{R}^{d},r)$ is a Hilbert space and assumption
(\ref{eq:Gassump}) ensures that $\mathcal{G}f\in L^{2}(\mathbb{R}^{d},r)$ for
all $f\in C_{b}^{2}(\mathbb{R}^{d})$. In $L^{2}(\mathbb{R}^{d},r)$, the basic
adjoint relationship in (\ref{eq:bar_density}) is equivalent to
\begin{equation}
\langle\mathcal{G}f,q\rangle=0\qquad\text{ for all }f\in C_{b}^{2}%
(\mathbb{R}^{d})\text{.} \label{eq:barref}%
\end{equation}
With a fixed reference density $r$, we need only compute the ratio function
$q$ by (\ref{eq:barref}). Once $q$ is obtained, we can compute the stationary
density via $g(x)=q(x)r(x)$ for $x\in\mathbb{R}^{d}$.

Let
\begin{equation}
H=\text{the closure of }\{\mathcal{G}f:f\in C_{b}^{2}(\mathbb{R}^{d})\}
\label{eq:H}%
\end{equation}
where the closure is taken in the norm in (\ref{eq:norm}). As a subspace of
$L^{2}(\mathbb{R}^{d},r)$, $H$ is orthogonal to $q$. Let $c$ be a constant
function and $c(x)=1$ for all $x\in\mathbb{R}^{d}$. Clearly, $c\in
L^{2}(\mathbb{R}^{d},r)$ but $c\notin H$ because%
\begin{equation}
\langle c,q\rangle=\int_{\mathbb{R}^{d}}g(x)\,\mathrm{d}x=1.
\label{eq:normalizing}%
\end{equation}
Let%
\begin{equation}
\bar{c}=\argmin_{f\in H}\left\Vert c-f\right\Vert \label{eq:gbar}%
\end{equation}
be the projection of $c$ onto $H$. Then, $c-\bar{c}$ must be orthogonal to
$H$. Assuming that Conjecture \ref{conj:signedMeasure} holds and $X$ has a
unique stationary density $g$, one must have $q=\kappa_{q}(c-\bar{c})$ for
some constant $\kappa_{q}\in\mathbb{R}$. By (\ref{eq:normalizing}), the
normalizing constant $\kappa_{q}$ satisfies%
\[
\kappa_{q}^{-1}=\langle c,c-\bar{c}\rangle=\langle c-\bar{c},c-\bar{c}%
\rangle+\langle\bar{c},c-\bar{c}\rangle=\left\Vert c-\bar{c}\right\Vert ^{2}.
\]
Hence, the ratio function is given by
\begin{equation}
q=\frac{c-\bar{c}}{\left\Vert c-\bar{c}\right\Vert ^{2}}. \label{eq:ratio}%
\end{equation}

\subsection{An approximate stationary density}

To compute $q$ by (\ref{eq:ratio}), we need first compute $\bar{c}$, the
projection of $c$ onto $H$. The space $H$ is linear and infinite-dimensional
(i.e., a basis of $H$ contains infinitely many functions). In general, solving
(\ref{eq:gbar}) in an infinite-dimensional space is impossible. In the
algorithm, we use a finite-dimensional subspace $H_{k}$ to approximate $H$.

Suppose that there exists a sequence of finite-dimensional subspaces
$\{H_{k}:k\in\mathbb{N}\}$ of $H$ such that $H_{k}\rightarrow H$ in
$L^{2}(\mathbb{R}^{d},r)$\ as $k\rightarrow\infty$. Here, $H_{k}\rightarrow H$
in $L^{2}(\mathbb{R}^{d},r)$\ means that for each $f\in H$, there exists a
sequence of functions $\{\varphi_{k}:k\in\mathbb{N}\}$ with $\varphi_{k}\in
H_{k}$ such that $\left\Vert \varphi_{k}-f\right\Vert \rightarrow0$ as
$k\rightarrow\infty$. Let
\[
\bar{c}_{k}=\argmin_{f\in H_{k}}\left\Vert c-f\right\Vert
\]
be the projection of $c$ onto $H_{k}$. By Proposition 7 of \cite{daihar92}, we
have the following approximation result.

\begin{proposition}
\label{prop:Convergence}Assume that Conjecture \ref{conj:signedMeasure} is
true. Then,%
\[
\left\Vert q_{k}-q\right\Vert \rightarrow0\qquad\text{as }k\rightarrow
\infty\text{,}%
\]
where $q_{k}=(c-\bar{c}_{k})/\left\Vert c-\bar{c}_{k}\right\Vert ^{2}$.
Furthermore, when the reference density $r$ is bounded,
\[
\int_{\mathbb{R}^{d}}(g_{k}(x)-g(x))^{2}\,\mathrm{d}x\rightarrow
0\qquad\text{as }k\rightarrow\infty\text{,}%
\]
where $g_{k}(x)=q_{k}(x)r(x)$ for each $x\in\mathbb{R}^{d}$.
\end{proposition}

As in \cite{daihar92}, we choose
\begin{equation}
H_{k}=\{\mathcal{G}f:f\in C_{k}\} \label{eq:Hk}%
\end{equation}
for some finite-dimensional space $C_{k}$. In Section~\ref{sec:FEM}, we will
discuss how to construct $C_{k}$ using a finite element method. For notational
convenience, we omit the subscript $k$ when $k$ is fixed. The
finite-dimensional functional space is thus denoted by $C$. Let $m_{C}$ be the
dimension of $C$ and $\{f_{i}:i=1,\ldots,m_{C}\}$ be a basis of $C$. We assume
that the family $\{\mathcal{G}f_{i}:i=1,\ldots,m_{C}\}$ is linearly
independent in $L^{2}(\mathbb{R}^{d},r)$. Then,%
\begin{equation}
\bar{c}_{k}=\sum_{i=1}^{m_{C}}u_{i}\mathcal{G}f_{i}\qquad\text{for some }%
u_{i}\in\mathbb{R}\text{ and }i=1,\ldots,m_{C}\text{.} \label{eq:gn}%
\end{equation}
Using the fact $\langle\mathcal{G}f_{i},c-\bar{c}_{k}\rangle=0$ for
$i=1,\ldots,m_{C}$, we obtain a system of linear equations%
\begin{equation}
Au=v \label{eq:equation}%
\end{equation}
where
\begin{equation}
A_{i\ell}=\langle\mathcal{G}f_{i},\mathcal{G}f_{\ell}\rangle,\qquad
u=(u_{1},\ldots,u_{m_{C}})^{\prime},\qquad v_{i}=\langle\mathcal{G}%
f_{i},c\rangle. \label{eq:coeff}%
\end{equation}
By the linear independence assumption, the $m_{C}\times m_{C}$\ matrix $A$ is
positive definite. Thus, $u=A^{-1}v$ is the unique solution to
(\ref{eq:equation}). Once the vector $u$ is obtained, we can compute the
projection $\bar{c}_{k}$ by (\ref{eq:gn}). Finally, the stationary density $g$
can be approximated via%
\[
g(x)\approx g_{k}(x)=r(x)\frac{c(x)-\bar{c}_{k}(x)}{\left\Vert c-\bar{c}%
_{k}\right\Vert ^{2}}\qquad\text{for each }x\in\mathbb{R}^{d}\text{.}%
\]

\subsection{A finite element method}

\label{sec:FEM}

In \cite{daihar92}, the authors employed multinominals of orders up to $k$ to
construct the space $C_{k}$. This choice appears to be numerically unstable.
The approximation error is significant when $k$ is small, say, $k\leq5$. As
$k$ increases, the round-off error in solving (\ref{eq:equation}) increases
and ultimately dominates the approximation error. Although their
implementation produces accurate estimates for the stationary means of SRBMs,
it sometimes produces poor estimates for the stationary distributions. In this
section, we construct a sequence of spaces $\{C_{k}:k\in\mathbb{N}\}$ using
the finite element method as in \cite{scdd02}. Because the state space in
\cite{scdd02} is bounded, neither a reference density nor state space
truncation is used there.

The state space of $X$ is unbounded in our setting. It is necessary to
truncate the state space to apply the finite element method. Let $\{K_{k}%
:k\in\mathbb{N}\}$ be a sequence of compact sets in $\mathbb{R}^{d}$. For each
$f\in C_{k}$, we assume that $f(x)=0$ for $x\in\mathbb{R}^{d}\setminus K_{k}$.
The subscript $k$ is omitted again when it is fixed and we use $K$ to denote
the compact support of the space $C$. In our implementation, we restrict $K$
to be a $d$-dimensional hypercube%
\begin{equation}
K=[-\zeta_{1},\xi_{1}]\times\cdots\times\lbrack-\zeta_{d},\xi_{d}],
\label{eq:K}%
\end{equation}
where both $\zeta_{j}$ and $\xi_{j}$ are positive constants for $j=1,\ldots,d$.

We partition $K$ into a finite number of subdomains. Such a partition is
called a \emph{mesh} and each subdomain is called a \emph{finite element}.
Since $K$ is a hypercube, it is natural to use a lattice mesh, where each
finite element is again a hypercube. In this case, each corner point of a
finite element is called a \emph{node}. In dimension $j=1,\ldots,d$, we divide
the interval $[-\zeta_{j},\xi_{j}]$ into $n_{j}$ subintervals by partition
points
\[
-\zeta_{j}=y_{j}^{0}<y_{j}^{1}<\cdots<y_{j}^{n_{j}}=\xi_{j}.
\]
Then, $K$ is divided into $\prod_{j=1}^{d}n_{j}$ finite elements. For future
reference, we label the nodes in the way that node $(i_{1},\ldots,i_{d})$
corresponds to spatial coordinate $(y_{1}^{i_{1}},\ldots,y_{d}^{i_{d}})$, and
define%
\[
h_{j}^{\ell}=y_{j}^{\ell+1}-y_{j}^{\ell}\qquad\text{for }\ell=0,\ldots
,n_{j}-1\text{ and }j=1,\ldots,d\text{.}%
\]
If $\Delta$ denotes such a mesh, we define
\[
\left\vert \Delta\right\vert =\max\{h_{j}^{\ell}:\ell=0,\ldots,n_{j}%
-1;\,j=1,\ldots,d\}
\]
and
\begin{equation}
\eta_{\Delta}=\max\bigg\{\frac{h_{j_{1}}^{\ell_{1}}}{h_{j_{2}}^{\ell_{2}}%
}:\ell_{1},\ell_{2}=0,\ldots,n_{j}-1;\,j_{1},j_{2}=1,\ldots,d;\,\,j_{1}\neq
j_{2}\bigg\}. \label{eq:eta}%
\end{equation}

The finite-dimensional space $C$ is generated using the above mesh. We use the
cubic Hermite basis functions to construct a basis of $C$, as in
\cite{scdd02}. The one-dimensional Hermite basis functions for $-1\leq z\leq1$
are given by%
\begin{equation}
\phi(z)=(\left\vert z\right\vert -1)^{2}(2\left\vert z\right\vert
+1)\qquad\text{and}\qquad\psi(z)=z(\left\vert z\right\vert -1)^{2}\text{.}
\label{eq:Hermit}%
\end{equation}
In dimension $j=1,\ldots,d$ and for $\ell=1,\ldots,n_{j}-1$, let%
\[
\phi_{j}^{\ell}(z)=%
\begin{cases}
\phi\bigg(\dfrac{z-y_{j}^{\ell}}{h_{j}^{\ell-1}}\bigg) & \text{if }y_{j}%
^{\ell-1}\leq z\leq y_{j}^{\ell}\text{,}\\
\phi\bigg(\dfrac{z-y_{j}^{\ell}}{h_{j}^{\ell}}\bigg) & \text{if }y_{j}^{\ell
}\leq z\leq y_{j}^{\ell+1}\text{,}\\
0 & \text{otherwise}%
\end{cases}
\]
and%
\[
\psi_{j}^{\ell}(z)=%
\begin{cases}
h_{j}^{\ell-1}\psi\bigg(\dfrac{z-y_{j}^{\ell}}{h_{j}^{\ell-1}}\bigg) &
\text{if }y_{j}^{\ell-1}\leq z\leq y_{j}^{\ell}\text{,}\\
h_{j}^{\ell}\psi\bigg(\dfrac{z-y_{j}^{\ell}}{h_{j}^{\ell}}\bigg) & \text{if
}y_{j}^{\ell}\leq z\leq y_{j}^{\ell+1}\text{,}\\
0 & \text{otherwise.}%
\end{cases}
\]
Let $x=(x_{1},\ldots,x_{d})^{\prime}$ be a vector in $K$. At node
$(i_{1},\ldots,i_{d})$, the basis functions of $C$ are the tensor-product
Hermite basis functions%
\begin{equation}
f_{i_{1},\ldots,i_{d},\chi_{1},\ldots,\chi_{d}}(x)=\prod_{j=1}^{d}%
g_{i_{j},\chi_{j}}(x_{j}) \label{eq:basisFunctions}%
\end{equation}
where $\chi_{j}$ is either $0$ or $1$ and
\[
g_{i_{j},\chi_{j}}(z)=%
\begin{cases}
\phi_{j}^{i_{j}}(z) & \text{if }\chi_{j}=0\text{,}\\
\psi_{j}^{i_{j}}(z) & \text{if }\chi_{j}=1\text{.}%
\end{cases}
\]
Therefore, each node has $2^{d}$ tensor-product basis functions and the space
$C$ has a total of
\begin{equation}
m_{C}=2^{d}\prod_{j=1}^{d}(n_{j}-1) \label{eq:Nn}%
\end{equation}
basis functions.

The space $C$ is not a subspace of $C_{b}^{2}(\mathbb{R}^{d})$. For the
one-dimensional Hermite basis functions in (\ref{eq:Hermit}), the second order
derivative of $\phi(z)$\ is not defined at $z=-1$ and $1$, and the second
order derivative of $\psi(z)$ is not defined at $z=-1$, $0$, and $1$. As a
consequence, there exists $f\in C$\ for which $\mathcal{G}f$ is not defined on
the boundaries of certain finite elements. Because such boundaries have
Lebesgue measure zero in $\mathbb{R}^{d}$, for each $f\in C$, we can find a
sequence of functions $\{\varphi_{i}:i\in\mathbb{N}\}$ in $C_{b}%
^{2}(\mathbb{R}^{d})$ such that $\left\Vert \mathcal{G}\varphi_{i}%
-\mathcal{G}f\right\Vert \rightarrow0$ as $i\rightarrow\infty$. Hence,
$H_{k}\subset H$ still holds for each $k$.

For the linear system (\ref{eq:equation}) to have a unique solution, the
family of functions
\[
\{\mathcal{G}f_{i_{1},\ldots,i_{d},\chi_{1},\ldots,\chi_{d}}:i_{j}%
=1,\ldots,n_{j}-1;\chi_{j}=0,1;j=1,\ldots d\}
\]
must be linearly independent in $L^{2}(\mathbb{R}^{d},r)$. The following
proposition provides sufficient conditions for the linear independence. Its
proof can be found in the appendix.

\begin{proposition}
\label{prop:independence} Let $\mathcal{G}$ be the generator of $X$ in
(\ref{eq:generator}) such that conditions (\ref{eq:Lipschitz}) and
(\ref{eq:uelliptic}) holds and all entries of $\Sigma$ are continuously
differentiable. Assume that $r(x)>0$ for all $x\in\mathbb{R}^{d}$. Then, the
family of functions
\[
\{\mathcal{G}f_{i_{1},\ldots,i_{d},\chi_{1},\ldots,\chi_{d}}:i_{j}%
=1,\ldots,n_{j}-1;\chi_{j}=0,1;j=1,\ldots d\}
\]
is linearly independent in $L^{2}(\mathbb{R}^{d},r)$, where $f_{i_{1}%
,\ldots,i_{d},\chi_{1},\ldots,\chi_{d}}$ is the basis function of $C$ given by
(\ref{eq:basisFunctions}). Consequently, the solution to the linear system
(\ref{eq:equation}) is unique.
\end{proposition}

Now let us consider a sequence of functional spaces $\{C_{k}:k\in\mathbb{N}%
\}$. Let $\Delta_{k}$ be the mesh for constructing $C_{k}$. We assume that the
mesh $\Delta_{k+1}$ is a refinement of $\Delta_{k}$, i.e., a node or an
interelement boundary in $\Delta_{k}$ is also a node or an interelement
boundary in $\Delta_{k+1}$. We further assume that such refinements are
\emph{regular}, i.e., for each $\eta_{\Delta_{k}}$ defined in (\ref{eq:eta}),
the set $\{\eta_{\Delta_{k}}:k\in\mathbb{N}\}$ is bounded. The next
proposition, along with Proposition \ref{prop:Convergence}, justifies the
proposed algorithm for computing the stationary distribution. We leave the
proof of Proposition \ref{prop:Hn} to the appendix, too.

\begin{proposition}
\label{prop:Hn} Let $\{\Delta_{k}:k\in\mathbb{N}\}$ be a sequence of lattice
meshes such that each $\Delta_{k+1}$ is a refinement of $\Delta_{k}$ and the
refinements are regular. Let $K_{k}$ be the $d$-dimensional finite hypercube
that is the domain of $\Delta_{k}$, and $C_{k}$ be the functional space
generated by $\Delta_{k}$ using the tensor-product Hermite basis functions in
(\ref{eq:basisFunctions}). Let $H$ be the infinite-dimensional space in
(\ref{eq:H}) and $H_{k}$ be the finite-dimensional space in (\ref{eq:Hk}),
where the generator $\mathcal{G}$\ satisfies (\ref{eq:Lipschitz}) and
(\ref{eq:Gassump}). Assume that
\[
\left\vert \Delta_{k}\right\vert \rightarrow0\qquad\text{and}\qquad
K_{k}\uparrow\mathbb{R}^{d}\qquad\text{as }k\rightarrow\infty\text{.}%
\]
Then,%
\[
H_{k}\rightarrow H\qquad\text{as }k\rightarrow\infty\text{.}%
\]

\end{proposition}

\section{Diffusion models for many-server queues}

\label{sec:DiffusionQueues}

In this section, we introduce $GI/Ph/n+GI$ queues and present two diffusion
models for such a queue with many servers. The two models differ in how the
patience time distribution is built into them. The patience time density at
zero is used in the first model, whereas the entire patience time distribution
is used in the second model.

\subsection{$GI/Ph/n+GI$ queues in the QED regime}

\label{sec:Queues}

We focus on a queue with many servers working in the QED regime. The QED
regime will be discussed shortly. In this queue, the service time distribution
is restricted to be phase-type. All positive-valued distributions can be
approximated by phase-type distributions.

Let $p$ be a $d$-dimensional nonnegative vector whose entries sum to one,
$\nu$ be a $d$-dimensional positive vector, and $P$ be a $d\times d$
sub-stochastic matrix. We assume that the diagonal entries of $P$ are zero and
$P$ is transient, namely, $I-P$ is invertible. Consider a continuous-time
Markov chain with $d+1$ phases (or states) where phases $1,\ldots,d$ are
transient and phase $d+1$ is absorbing. For $j=1,\ldots,d$, the Markov chain
starts in phase $j$ with probability $p_{j}$. The amount of time it stays in
phase $j$ is exponentially distributed with mean $1/\nu_{j}$. When it leaves
phase $j$, the Markov chain enters phase $\ell=1,\ldots,d$ with probability
$P_{j\ell}$ or enters phase $d+1$ with probability $1-\sum_{\ell=1}%
^{d}P_{j\ell}$. The \emph{phase-type distribution} with parameters $(p,\nu,P)$
is the distribution of time from starting until absorption in phase $d+1$ for
the above Markov chain. In particular, when $P$ is a zero matrix, the
associated phase-type distribution is a \emph{hyperexponential distribution}
with $d$ phases.

In a $GI/Ph/n+GI$ queue, there are $n$ identical servers working in parallel.
The customer arrival process is a renewal process. Upon arrival, a customer
enters service immediately if an idle server is available. Otherwise, he waits
in a buffer with infinite waiting room that holds a first-in-first-out (FIFO)
queue. The service times form a sequence of iid random variables, following a
phase-type distribution. When a server finishes serving a customer, the server
takes the leading customer from the waiting buffer. When the buffer is empty,
the server begins to idle. Each customer has a patience time. The patience
times are iid following a general distribution. When a customer's waiting time
in queue exceeds his patience time, the customer abandons the system with no service.

Let $\lambda$ be the arrival rate and $1/\mu$ be the mean service time. The
system is assumed to operate in the QED regime, i.e., both the arrival rate
$\lambda$ and the number of servers $n$ are large, while the traffic intensity
$\rho=\lambda/(n\mu)$ is close to one. Because customer abandonment is
allowed, it is not necessary to assume $\rho<1$ for the system to reach a
steady state. For future purposes, we put
\begin{equation}
\beta=\sqrt{n}(1-\rho). \label{eq:beta}%
\end{equation}

Assume that the phase-type service time distribution has parameters
$(p,\nu,P)$. Each service time can be decomposed into a number of phases. When
a customer is in service, he must be in one of the $d$ phases. Let $Z_{j}(t)$
denote the number of customers in phase $j$ service at time $t$. In
steady-state, one expects that the customers in service are distributed among
the $d$ phases following a distribution $\gamma$, given by
\begin{equation}
\gamma=\mu R^{-1}p\qquad\text{and}\qquad R=(I-P^{\prime})\diag(\nu)\text{.}
\label{eq:R}%
\end{equation}
One can check that $\sum_{j=1}^{d}\gamma_{j}=1$ and $\gamma_{j}$ is
interpreted to be the fraction of phase $j$ service load on the $n$ servers.

Suppose that all customers, including those initial customers waiting in the
buffer at time zero, sample their first service phases following distribution
$p$ upon arrival. One can stratify customers in the waiting buffer according
to their first service phases. For $j=1,\ldots,d$, we use $W_{j}(t)$ to denote
the number of waiting customers at time $t$ whose service begins with phase
$j$. Then,
\begin{equation}
Y_{j}(t)=Z_{j}(t)+W_{j}(t) \label{eq:Yj}%
\end{equation}
is the number of phase $j$ customers in the system, either waiting or in
service. Let $Y(t)$ be the corresponding $d$-dimensional random vector and%
\begin{equation}
\tilde{Y}(t)=\frac{1}{\sqrt{n}}(Y(t)-n\gamma). \label{eq:Ydiffusion}%
\end{equation}
In each diffusion model, the process $\tilde{Y}=\{\tilde{Y}(t):t\geq0\}$ is
approximated by a $d$-dimensional diffusion process.

\subsection{System equation}

\label{sec:SystemEquations}

The $GI/Ph/n+GI$ queue is driven by several primitive processes. Let
$E=\{E(t):t\geq0\}$ be the arrival process, where $E(t)$ is the number of
arrivals by time $t$. For $j=1,\ldots,d$, let $S_{j}=\{S_{j}(t):t\geq0\}$ be a
Poisson process with rate $\nu_{j}$, and $\phi_{j}=\{\phi_{j}(i):i\in
\mathbb{N}\}$ be a sequence of iid $d$-dimensional random vectors such that
$\phi_{j}(i)$ takes $e^{\ell}$ with probability $P_{j\ell}$ and takes a zero
vector with probability $1-\sum_{\ell=1}^{d}P_{j\ell}$. Similarly, let
$\phi_{0}=\{\phi_{0}(i):i\in\mathbb{N}\}$ be a sequence of iid $d$-dimensional
random vectors such that $\phi_{0}(i)$ takes $e^{\ell}$ with probability
$p_{\ell}$. For $j=0,\ldots,d$, define the routing process $\Phi_{j}%
=\{\Phi_{j}(k):k\in\mathbb{N}\}$ by
\[
\Phi_{j}(k)=\sum_{i=1}^{k}\phi_{j}(i).
\]
We assume that $Y(0),E,S_{1},\ldots,S_{d},\Phi_{0},\ldots,\Phi_{d}$ are
mutually independent.

For $j=1,\ldots,d$, let $T_{j}(t)$ be the cumulative amount of service effort
received by customers in phase $j$ service by time $t$. Clearly,%
\begin{equation}
T_{j}(t)=\int_{0}^{t}Z_{j}(s)\,\mathrm{d}s\qquad\text{for }t\geq0\text{.}
\label{eq:T}%
\end{equation}
Thus, $S_{j}(T_{j}(t))$ is equal in distribution to the cumulative number of
phase $j$ service completions by time $t$. (For more details, please refer to
Section~4.1 of \cite{DaiHeTezcan10} on a perturbed system.) Let $L_{j}(t)$ be
the cumulative number of phase $j$ customers who have abandoned the system by
time $t$, and $L(t)$ be the corresponding $d$-dimensional vector. One can
check that the process $Y=\{Y(t):t\geq0\}$ satisfies the following equation%
\begin{equation}
Y(t)=Y(0)+\Phi_{0}(E(t))+\sum_{j=1}^{d}\Phi_{j}(S_{j}(T_{j}(t)))-S(T(t))-L(t),
\label{eq:sysequ}%
\end{equation}
where $S(T(t))=(S_{1}(T_{1}(t)),\ldots,S_{d}(T_{d}(t)))^{\prime}$.

To derive the diffusion models, consider a scaled version of (\ref{eq:sysequ}%
). We define several scaled processes by
\begin{align*}
&  \tilde{E}(t) =\frac{1}{\sqrt{n}}(E(t)-\lambda t),\qquad\tilde{S}%
(t)=\frac{1}{\sqrt{n}}(S(nt)-n\nu t),\\
&  \tilde{Z}(t) =\frac{1}{\sqrt{n}}(Z(t)-n\gamma,\qquad\tilde{L}(t) =\frac
{1}{\sqrt{n}}L(t),)\\
&  \tilde{\Phi}_{0}(t) =\frac{1}{\sqrt{n}}\sum_{i=1}^{\left\lfloor
nt\right\rfloor }(\phi_{0}(i)-p),\qquad\tilde{\Phi}_{j}(t)=\frac{1}{\sqrt{n}%
}\sum_{i=1}^{\left\lfloor nt\right\rfloor }(\phi_{j}(i)-p^{j})
\end{align*}
for $t\geq0$ and $j=1,\ldots,d$, where $p^{j}$ is the $j$th column of
$P^{\prime}$. By (\ref{eq:beta})--(\ref{eq:T}), the dynamical equation in
(\ref{eq:sysequ}) turns out to be%
\begin{equation}%
\begin{split}
\tilde{Y}(t)  &  =\tilde{Y}(0)-\beta\mu pt+p\tilde{E}(t)+\tilde{\Phi}%
_{0}\Big(\frac{E(t)}{n}\Big)\\
&  \quad+\sum_{j=1}^{d}\tilde{\Phi}_{j}\Big(\frac{S_{j}(T_{j}(t))}%
{n}\Big)-(I-P^{\prime})\tilde{S}\Big(\frac{T(t)}{n}\Big)-R\int_{0}^{t}%
\tilde{Z}(s)\,\mathrm{d}s-\tilde{L}(t).
\end{split}
\label{eq:sys}%
\end{equation}

\subsection{Diffusion models}

\label{sec:models}

In both diffusion models, we replace the scaled primitive processes in
(\ref{eq:sys}) by certain Brownian motions. These approximations can be
justified by the functional central limit theorem. When the number of servers
$n$ is large, the corresponding diffusion process in each model can be proved
close to $\tilde{Y}$ via a continuous map. Please refer to
\cite{DaiHeTezcan10} for related convergence results.

Let $B_{E}$ be a one-dimensional driftless Brownian motion with variance
$\lambda c_{a}^{2}/n$, where $c_{a}^{2}$ is the squared coefficient of
variation for the interarrival time distribution. Let $B_{0},\ldots,B_{d}$,
and $B_{S}$ are $d$-dimensional driftless Brownian motions with covariance
matrices $H^{0},\ldots,H^{d}$, and $\operatorname*{diag}(\nu)$, respectively,
where
\[
H_{k\ell}^{0}=%
\begin{cases}
p_{k}(1-p_{\ell}) & \text{if }k=\ell\text{,}\\
-p_{k}p_{\ell} & \text{otherwise}%
\end{cases}
\]
and
\[
H_{k\ell}^{j}=%
\begin{cases}
P_{jk}(1-P_{j\ell}) & \text{if }k=\ell\text{,}\\
-P_{jk}P_{j\ell} & \text{otherwise}%
\end{cases}
\]
for $j=1,\ldots,d$. We assume that $\tilde{Y}(0),B_{E},B_{0},\ldots
,B_{d},B_{S}$ are mutually independent. In the diffusion models, the above
Brownian motions take the places of the scaled primitive processes $\tilde
{E},\tilde{\Phi}_{0},\ldots,\tilde{\Phi}_{d},\tilde{S}$, respectively. Let
$Q(t)$ be the queue length or the number of waiting customers at time $t$ and
\[
\tilde{Q}(t)=\frac{1}{\sqrt{n}}Q(t).
\]
Then, $Q(t)=(e^{\prime}Y(t)-n)^{+}$ or equivalently,
\begin{equation}
\tilde{Q}(t)=(e^{\prime}\tilde{Y}(t))^{+}. \label{eq:Q}%
\end{equation}
When $n$ is large, these waiting customers are approximately distributed among
the $d$ phases according to distribution $p$ (see Lemma 2 of
\cite{DaiHeTezcan10}), i.e.,%
\[
W_{j}(t)\approx p_{j}Q(t)\qquad\text{for }j=1,\ldots,d\text{.}%
\]
It follows from (\ref{eq:Yj}) that%
\[
Z(t)\approx Y(t)-pQ(t).
\]
By (\ref{eq:Ydiffusion}) and (\ref{eq:Q}), this approximation has a scaled
version%
\begin{equation}
\tilde{Z}(t)\approx\tilde{Y}(t)-p(e^{\prime}\tilde{Y}(t))^{+}.
\label{eq:Ztild}%
\end{equation}
Let $G(t)$ be the cumulative number of abandoned customers by time $t$ and
\[
\tilde{G}(t)=\frac{1}{\sqrt{n}}G(t).
\]
These abandoned customers are also approximately distributed among the $d$
phases by distribution $p$, i.e.,
\begin{equation}
\tilde{L}(t)\approx p\tilde{G}(t). \label{eq:LG}%
\end{equation}
We also exploit the following approximations%
\begin{equation}
\frac{E(t)}{n}\approx\frac{\lambda t}{n}=\rho\mu t,\qquad\frac{T(t)}{n}%
\approx(\rho\wedge1)\gamma t,\qquad\frac{S_{j}(T_{j}(t))}{n}\approx(\rho
\wedge1)\nu_{j}\gamma_{j}t. \label{eq:appr}%
\end{equation}
The approximations in (\ref{eq:Ztild})--(\ref{eq:appr}) are used in both
diffusion models. These two models differ only in how to approximate the
scaled abandonment process $\tilde{G}=\{\tilde{G}(t):t\geq0\}$.

\subsubsection{Diffusion model using the patience time density at zero}

\label{sec:PatienceZero}

In the first diffusion model, the patience time distribution is used only
through its density at zero when approximating $\tilde{G}$. Let $F$ be the
distribution function of the patience times. We assume that
\begin{equation}
F(0)=0\qquad\text{and}\qquad\alpha=\lim_{t\downarrow0}\frac{F(t)}{t}%
<\infty\text{.} \label{eq:patiencecond}%
\end{equation}
Thus, $\alpha$ is the patience time density at zero. In this model, $\tilde
{G}$ is approximated by%
\begin{equation}
\tilde{G}(t)\approx\alpha\int_{0}^{t}(e^{\prime}\tilde{Y}(s))^{+}%
\,\mathrm{d}s\qquad\text{for }t\geq0\text{.} \label{eq:abandonment}%
\end{equation}
When $\alpha=0$, this approximation yields $\tilde{G}(t)\approx0$ for all
$t\geq0$. In this case, the diffusion model is for a $GI/Ph/n$ queue without
abandonment. When $\alpha>0$, the intuition of (\ref{eq:abandonment}) is as
follows. For a many-server queue in the QED regime, the queue length is
typically in the order of $O(n^{1/2})$. As a result, each customer's waiting
time should be in the order of $O(n^{-1/2})$, which is relatively short
because $n$ is large. At time $s$, a waiting customer will abandon the system
during the next $\delta$ time units with probability $\alpha\delta$. Hence,
the instantaneous abandonment rate of the system is approximately $\alpha
Q(s)$. It follows that%
\[
G(t)\approx\alpha\int_{0}^{t}Q(s)\,\mathrm{d}s,
\]
which, along with (\ref{eq:Ydiffusion}) and (\ref{eq:Q}), leads to the
approximation in (\ref{eq:abandonment}). This relationship can be justified by
Theorem 1 of \cite{DaiHe10}. As pointed out by \cite{DaiHe10} and
\cite{MandelbaumMomcilovic09}, the performance of a many-server queue in the
QED regime is insensitive to the patience time distribution as long as its
density $\alpha$ at zero is fixed and positive.

In the dynamical equation (\ref{eq:sys}), when the scaled primitive processes
are replaced by appropriate Brownian motions and the approximations in
(\ref{eq:Ztild})--(\ref{eq:abandonment}) are employed, we obtain the following
stochastic differential equation%
\begin{equation}%
\begin{split}
X(t)  &  =X(0)-\beta\mu pt+pB_{E}(t)+B_{0}(\rho\mu t)+\sum_{j=1}^{d}%
B_{j}((\rho\wedge1)\nu_{j}\gamma_{j}t)\\
&  \quad-(I-P^{\prime})B_{S}((\rho\wedge1)\gamma t)-R\int_{0}^{t}%
(X(s)-p(e^{\prime}X(s))^{+})\,\mathrm{d}s\\
&  \quad-p\alpha\int_{0}^{t}(e^{\prime}X(s))^{+}\,\mathrm{d}s
\end{split}
\label{eq:model1}%
\end{equation}
where we take $X(0)=\tilde{Y}(0)$. We may write (\ref{eq:model1}) into the
standard form
\[
X(t)=X(0)+\int_{0}^{t}b(X(s))\,\mathrm{d}s+\int_{0}^{t}\sigma
(X(s))\,\mathrm{d}B(s),
\]
where for each $x\in\mathbb{R}^{d}$, the drift coefficient $b$ is
\begin{equation}
b(x)=-\beta\mu p-R(x-p(e^{\prime}x)^{+})-p\alpha(e^{\prime}x)^{+},
\label{eq:drift1}%
\end{equation}
the diffusion coefficient $\sigma$ is a $d\times d$ constant matrix satisfying%
\begin{equation}%
\begin{split}
\Sigma(x)  &  =\sigma(x)\sigma^{\prime}(x)\\
&  =\rho\mu(c_{a}^{2}pp^{\prime}+H^{0})\\
&  \quad+(\rho\wedge1)\bigg(\sum_{j=1}^{d}\nu_{j}\gamma_{j}H^{j}+(I-P^{\prime
})\operatorname*{diag}(\nu)\operatorname*{diag}(\gamma)(I-P)\bigg),
\end{split}
\label{eq:covariance}%
\end{equation}
and $B$ is a $d$-dimensional standard Brownian motion. One can check that
$\Sigma(x)$ is positive definite and thus satisfies (\ref{eq:uelliptic}). The
drift coefficient $b$ in (\ref{eq:drift1}) is a piecewise linear function of
$x$. Both $b$ and $\sigma$ are Lipschitz continuous. Therefore, a strong
solution to (\ref{eq:model1}) exists and is known as a $d$-dimensional
piecewise OU process. In this model, the diffusion process $X$ depends on the
patience time distribution only through its density\ at zero. When using the
proposed algorithm to solve the stationary density, it follows from
Proposition~\ref{prop:independence} that the linear system (\ref{eq:equation})
has a unique solution.

If we replace $\rho$ by one in (\ref{eq:model1}), the resulting diffusion
process turns out to be the diffusion limit for $G/Ph/n+GI$ queues in Theorem
2 of \cite{DaiHeTezcan10}. This limit process is the strong solution of the
following stochastic differential equation%
\begin{equation}%
\begin{split}
\check{X}(t)  &  =\check{X}(0)-\beta\mu pt+pB_{E}(t)+B_{0}(\mu t)+\sum
_{j=1}^{d}B_{j}(\nu_{j}\gamma_{j}t)\\
&  \quad-(I-P^{\prime})B_{S}(\gamma t)-R\int_{0}^{t}(\check{X}(s)-p(e^{\prime
}\check{X}(s))^{+})\,\mathrm{d}s\\
&  \quad-p\alpha\int_{0}^{t}(e^{\prime}\check{X}(s))^{+}\,\mathrm{d}s.
\end{split}
\label{eq:diffusionlim}%
\end{equation}
Since $\rho$ is close to one in the current setting, the above limit process
justifies the diffusion model in (\ref{eq:model1}).

\subsubsection{Diffusion model using patience time hazard rate scaling}

\label{sec:HazardRate}

When the patience time distribution does not have a density at zero, the
diffusion model in (\ref{eq:model1}) fails to exist. When $\alpha=0$ and
$\rho>1$, the diffusion process $X$ in (\ref{eq:model1}) does not have a
stationary distribution. In this case, the model cannot be a satisfactory
approximation of the many-server queue, as the queue may have a stationary
distribution thanks to customer abandonment. It is also demonstrated in
\cite{ReedTezcan09} that when the density near zero rapidly changes, the
system performance can be sensitive to the patience time distribution in a
neighborhood of the origin. In this case, using the patience time density at
zero solely may not yield adequate approximation to the queue. Our second
diffusion model exploits the idea of scaling the patience time hazard rate
function, which was first proposed in \cite{ReedWard08} for single-server
queues and was recently extended to many-server queues in \cite{ReedTezcan09}.

In this model, we assume that the patience time distribution $F$ satisfies
\[
F(0)=0
\]
and it has a bounded hazard rate function $h$, given by%
\[
h(t)=\frac{f_{F}(t)}{1-F(t)}\qquad\text{for }t\geq0\text{,}%
\]
where $f_{F}$ is the density of $F$. With the hazard rate function, $F$ can be
written by%
\[
F(t)=1-\exp\Big(-\int_{0}^{t}h(s)\,\mathrm{d}s\Big)\qquad\text{for }%
t\geq0\text{.}%
\]
In the second diffusion model, the scaled abandonment process $\tilde{G}$ is
approximated by%
\begin{equation}
\tilde{G}(t)\approx\int_{0}^{t}\int_{0}^{(e^{\prime}\tilde{Y}(s))^{+}%
}h\Big(\frac{\sqrt{n}u}{\lambda}\Big)\,\mathrm{d}u\,\mathrm{d}s\qquad\text{for
}t\geq0\text{.} \label{eq:hazaband}%
\end{equation}
The entire patience time distribution is built into the approximation through
its hazard rate function. The intuition of the hazard rate scaling
approximation was explained in \cite{ReedWard08}: Consider the $Q(s)$ waiting
customers in the buffer at time $s$. In general, only a small fraction of
customers can abandon the system when the queue is working in the QED regime.
Then by time $s$, the $i$th customer from the back of the queue has been
waiting around $i/\lambda$ time units. Approximately, this customer will
abandon the queue during the next $\delta$ time units with probability
$h(i/\lambda)\delta$. It follows that for the system, the instantaneous
abandonment rate at time $s$ is close to $\sum_{i=1}^{Q(s)}h(i/\lambda)$. By
(\ref{eq:Ydiffusion}) and (\ref{eq:Q}), the scaled abandonment rate can be
approximated by%
\begin{equation}
\frac{1}{\sqrt{n}}\sum_{i=1}^{Q(s)}h\Big(\frac{i}{\lambda}\Big)\approx\int%
_{0}^{\tilde{Q}(s)}h\Big(\frac{\sqrt{n}u}{\lambda}\Big)\,\mathrm{d}u=\int%
_{0}^{(e^{\prime}\tilde{Y}(s))^{+}}h\Big(\frac{\sqrt{n}u}{\lambda
}\Big)\,\mathrm{d}u, \label{eq:arate}%
\end{equation}
from which (\ref{eq:hazaband}) follows. Note that the arrival rate $\lambda$
is in the order of $O(n)$ and $Q(s)$ is in the order of $O(n^{1/2})$. The
patience time distribution in a small neighborhood of zero, not just its
density at zero, is considered in the instantaneous abandonment rate in
(\ref{eq:arate}). Hence, the hazard rate scaling approximation in
(\ref{eq:hazaband}) is more accurate than that in (\ref{eq:abandonment}). This
approximation can be justified for $GI/M/n+GI$ queues by Propositions~9.1
and~9.2 in \cite{ReedTezcan09}. With minor modifications to the proofs, these
two propositions can be extended to $GI/Ph/n+GI$ queues.

Let $m$ be a nonnegative integer. Suppose that the hazard rate function $h$ is
$m$ times continuously differentiable in a neighborhood of zero. By Taylor's
theorem,%
\[
h(z)\approx h(0)+\sum_{\ell=1}^{m}h^{(\ell)}(0)\frac{z^{\ell}}{\ell!}%
\]
for $z>0$ small enough, where $h^{(\ell)}$ is the $\ell$th order derivative of
$h$. In this case, the approximation in (\ref{eq:hazaband}) turns out to be%
\[
\tilde{G}(t)\approx h(0)\int_{0}^{t}(e^{\prime}\tilde{Y}(s))^{+}%
\,\mathrm{d}s+\sum_{\ell=1}^{m}\frac{n^{\ell/2}h^{(\ell)}(0)}{\lambda^{\ell
}(\ell+1)!}\int_{0}^{t}((e^{\prime}\tilde{Y}(s))^{+})^{\ell+1}\,\mathrm{d}s.
\]
Because $h(0)$ is identical to the patience time density at zero, the
approximation in (\ref{eq:abandonment}) can be regarded as the zeroth degree
Taylor's approximation of (\ref{eq:hazaband}). When the patience times are
exponentially distributed, the hazard rate function is constant and the two
approximations in (\ref{eq:abandonment}) and (\ref{eq:hazaband}) are
identical. See Section~4 of \cite{ReedTezcan09} for more discussion.

Using the Brownian motion replacement and the approximations in
(\ref{eq:Ztild})--(\ref{eq:appr}) and (\ref{eq:hazaband}), we obtain the
second diffusion model for the $GI/Ph/n+GI$ queue, given by%

\begin{equation}%
\begin{split}
X(t)  &  =X(0)-\beta\mu pt+pB_{E}(t)+B_{0}(\rho\mu t)+\sum_{j=1}^{d}%
B_{j}((\rho\wedge1)\nu_{j}\gamma_{j}t)\\
&  \quad-(I-P^{\prime})B_{S}((\rho\wedge1)\gamma t)-R\int_{0}^{t}%
(X(s)-p(e^{\prime}X(s))^{+})\,\mathrm{d}s\\
&  \quad-p\int_{0}^{t}\int_{0}^{(e^{\prime}X(s))^{+}}h\Big(\frac{\sqrt{n}%
u}{\lambda}\Big)\,\mathrm{d}u\,\mathrm{d}s.
\end{split}
\label{eq:model2}%
\end{equation}
The diffusion process $X$ in (\ref{eq:model2}) has the same diffusion
coefficient $\sigma$ as in the first model (\ref{eq:model1}). Its drift
coefficient $b$ is%
\begin{equation}
b(x)=-\beta\mu p-R(x-p(e^{\prime}x)^{+})-p\int_{0}^{(e^{\prime}x)^{+}%
}h\Big(\frac{\sqrt{n}u}{\lambda}\Big)\,\mathrm{d}u\qquad\text{for }%
x\in\mathbb{R}^{d}\text{.} \label{eq:drift2}%
\end{equation}
Because $h$ is bounded, the drift coefficient $b$ is Lipschitz continuous and
the stochastic differential equation (\ref{eq:model2}) has a strong solution.
By Proposition~\ref{prop:independence}, the solution to the linear system
(\ref{eq:equation}) is unique when we use the proposed algorithm to solve the
stationary density of this diffusion model. Comparing (\ref{eq:model1}) and
(\ref{eq:model2}), one can see that the two models differ only in how the
patience time distribution is incorporated. Because a more accurate
approximation is used for the abandonment process, the second model can
provide a better approximation for the queue.

\section{ Choosing a reference density}

\label{sec:ReferenceDensity}

The reference density controls the convergence of the proposed algorithm. In
this section, we discuss how to choose appropriate reference densities for the
diffusion models. Some considerations are as follows.

First, to be a reference density, a candidate function $r$ must satisfy
(\ref{eq:ref}) even though the stationary density $g$ is unknown. The second
condition in (\ref{eq:ref}) requires that $r$ have a comparable or slower
decay rate than $g$. When $g$ is bounded, its decay rate is sufficient to
determine a function $r$ that satisfies (\ref{eq:ref}).

Second, the most computational effort in the algorithm is constructing and
solving the system of linear equations (\ref{eq:equation}). As demonstrated by
Proposition \ref{prop:Hn}, the finite-dimensional $H_{k}$ approximates the
infinite-dimensional $H$ better as $k$ increases, thus reducing the
approximation error. On the other hand, as the dimension of $H_{k}$ increases,
constructing and solving (\ref{eq:equation}) requires more computation time
and memory space. The condition number of the matrix $A$ in (\ref{eq:equation}%
) also gets worse as the dimension of $H_{k}$ becomes large. This yields
higher round-off error. A \textquotedblleft good\textquotedblright\ reference
density should balance the approximation error and the round-off error. With
such a reference density, it is possible to have small approximation error
even if the dimension of $H_{k}$ is moderate.

Intuitively, when $r$ is \textquotedblleft close\textquotedblright\ to the
stationary density $g$, both the ratio function $q$ and the projection
$\bar{c}$ are close to constant. We can thus expect that a space $H_{k}$ with
a moderate dimension is able to produce a satisfactory approximation. All
these observations motivate us to explore the tail behavior of a diffusion model.

\subsection{Tail behavior}

\label{sec:Tail}

Let us focus on the diffusion limit in (\ref{eq:diffusionlim}). Assume that
the piecewise OU process $\check{X}$ is positive recurrent and has a
stationary distribution. Let $\check{X}(\infty)$ be the corresponding
$d$-dimensional random vector in steady state. Since $\rho$ is close to one,
the tail behavior of the diffusion process $X$ in (\ref{eq:model1}) is
expected to be comparable to that of the limit diffusion process $\check{X}$
in (\ref{eq:diffusionlim}).

To explore the tail behavior of $\check{X}(\infty)$, consider a sequence of
$GI/GI/n+GI$ queues in the QED regime. If all patience times are infinite, the
queues turn out to be $GI/GI/n$ queues without customer abandonment. In each
queue, the service times are iid following a general distribution. We assume
that these queues, each indexed by the number of servers $n$, have the same
service time distribution. Let $\lambda_{n}$ be the arrival rate of the $n$th
system. To mathematically define the QED regime, we assume that%
\begin{equation}
\lim_{n\rightarrow\infty}\frac{\lambda_{n}}{n}>0 \label{eq:lambda}%
\end{equation}
and
\begin{equation}
\lim_{n\rightarrow\infty}\sqrt{n}(1-\rho_{n})=\check{\beta}\qquad\text{for
some }\check{\beta}\in\mathbb{R}\text{,} \label{eq:QED}%
\end{equation}
where $\rho_{n}=\lambda_{n}/(n\mu)$ is the traffic intensity of the $n$th system.

Assume that all these queues are in steady state. Let $N_{n}(\infty)$ be the
stationary number of customers in the $n$th system and%
\[
\tilde{N}_{n}(\infty)=\frac{1}{\sqrt{n}}(N_{n}(\infty)-n).
\]
For $GI/GI/n$ queues in the QED regime, the limit queue length in steady state
was studied in \cite{GamarnikMomcilovic08}, where the service time
distribution is assumed to be lattice-valued on a finite support. The authors
first showed that $\tilde{N}_{n}(\infty)$ weakly converges to a random
variable $\check{N}(\infty)$ as $n$ goes to infinity, and then proved that%
\begin{equation}
\lim_{z\rightarrow\infty}\frac{1}{z}\log\mathbb{P}[\check{N}(\infty
)>z]=-\frac{2\check{\beta}}{c_{a}^{2}+c_{s}^{2}}, \label{eq:GGnTail}%
\end{equation}
where $c_{a}^{2}$ and $c_{s}^{2}$ are the squared coefficients of variation of
the interarrival and the service time distributions, respectively. In
(\ref{eq:GGnTail}), the decay rate does not depend on the service time
distribution beyond its first two moments. Recently, this result has been
extended in \cite{GamarnikGoldberg11} to $GI/GI/n$ queues with a general
service time distribution.

When $\alpha=0$ and $d=1$, the limit diffusion process $\check{X}$\ in
(\ref{eq:diffusionlim}) is for $GI/M/n$ queues without customer abandonment.
In this case, the service time distribution is exponential and $\check
{N}(\infty)=\check{X}(\infty)$. It was proved in \cite{HalWhi81} that the
stationary density of $\check{X}(\infty)$\ has a closed-form expression%
\begin{equation}
\check{g}(z)=%
\begin{cases}
a_{1}\exp\bigg(-\dfrac{(z+\check{\beta})^{2}}{1+c_{a}^{2}}\bigg) & \text{if
}z<0,\\
a_{2}\exp\Big(-\dfrac{2\check{\beta}z}{1+c_{a}^{2}}\Big) & \text{if }z\geq0,
\end{cases}
\label{eq:GMn}%
\end{equation}
where $a_{1}$ and $a_{2}$ are normalizing constants making $\check{g}$
continuous at zero. The decay rate of $\check{g}$ in (\ref{eq:GMn}) is
consistent with (\ref{eq:GGnTail}). Both formulas suggest that $\check
{N}(\infty)$ has an exponential tail on the right side.

For a $GI/GI/n+GI$ queue with many servers and customer abandonment, the
limiting tail behavior of $\tilde{N}_{n}(\infty)$ remains unknown except for
very simple cases. When $\alpha>0$ and $d=1$, the limit diffusion process
$\check{X}$ in (\ref{eq:diffusionlim}) is a one-dimensional piecewise OU
process. It admits a piecewise normal stationary density%
\begin{equation}
\check{g}(z)=%
\begin{cases}
a_{3}\exp\bigg(-\dfrac{(z+\check{\beta})^{2}}{1+c_{a}^{2}}\bigg) & \text{if
}z<0,\\
a_{4}\exp\bigg(-\dfrac{\alpha(z+\alpha^{-1}\mu\check{\beta})^{2}}{\mu
(1+c_{a}^{2})}\bigg) & \text{if }z\geq0,
\end{cases}
\label{eq:GMnGI}%
\end{equation}
where $a_{3}$ and $a_{4}$ are normalizing constants that make $\check{g}$
continuous at zero. See \cite{BrowneWhitt} for more details.

Observing (\ref{eq:GGnTail}) and (\ref{eq:GMnGI}), we conjecture that for a
sequence of $GI/GI/n+GI$ queues in the QED regime, the limiting tail behavior
of $\tilde{N}_{n}(\infty)$ depends on the service time distribution only
through its first two moments, and on the patience time distribution only
through its density at zero.

\begin{conjecture}
\label{conj:TailModel1} Consider a sequence of $GI/GI/n+GI$ queues that
satisfies (\ref{eq:patiencecond}), (\ref{eq:lambda}), and (\ref{eq:QED}).
Assume that the patience time distribution has a positive density at zero,
i.e., $\alpha>0$ in (\ref{eq:patiencecond}). Assume further that the
interarrival and the service time distributions satisfy the $T_{0}$
assumptions (i)--(iii) in Section~2.1 of \cite{GamarnikGoldberg11}. Then, (a)
$N_{n}(\infty)$ exists for each $n$; (b) the sequence of random variables
$\{\tilde{N}_{n}(\infty):n\in\mathbb{N}\}$ weakly converges to a random
variable $\check{N}(\infty)$; (c) $\check{N}(\infty)$ satisfies
\[
\lim_{z\rightarrow\infty}\frac{1}{z^{2}}\log\mathbb{P}[\check{N}%
(\infty)>z]=-\frac{\alpha}{\mu(c_{a}^{2}+c_{s}^{2})}.
\]

\end{conjecture}

The intuition below may help understand why the conjectured decay rate must be
Gaussian. When $\check{N}(\infty)>z$ for some $z>0$, there are more than
$n^{1/2}z$ waiting customers in the queue correspondingly, and each waiting
customer is \textquotedblleft racing\textquotedblright\ to abandon the system.
At any time, the instantaneous abandonment rate is approximately proportional
to the queue length. In such a system, the customer departure process,
including both service completions and customer abandonments, behaves as if
the system is a queue with infinite servers. Thus, one can expect that the
tail of the limit queue length is Gaussian, which decays much faster than an
exponential tail for queues without abandonment.

\subsection{Reference densities for model (\ref{eq:model1})}

\label{sec:RefTrunc}

For $GI/Ph/n+GI$ queues, the limit diffusion process $\check{X}$ in
(\ref{eq:diffusionlim}) satisfies
\[
\check{N}(\infty)=e^{\prime}\check{X}(\infty).
\]
The discussion in Section~\ref{sec:Tail} gives ample evidence of the tail
behavior $\mathbb{P}[\check{N}(\infty)>z]$ as $z\rightarrow\infty$. Although
the left tail $\mathbb{P}[\check{N}(\infty)<-z]$ as $z\rightarrow\infty$
remains unknown when $d>1$, our numerical experiments suggest that this tail
is not sensitive to the service time distribution beyond its mean. Thus, we
use the left tail for a queue with an exponential service time distribution to
construct the reference density. We propose to use a product reference density%
\begin{equation}
r(x)=\prod_{j=1}^{d}r_{j}(x_{j})\qquad\text{for }x\in\mathbb{R}^{d}\text{.}
\label{eq:reference}%
\end{equation}

When $\alpha=0$ and $\rho<1$ in (\ref{eq:model1}), there is no abandonment in
the queue. Based on (\ref{eq:GGnTail}) and (\ref{eq:GMn}), we choose
\begin{equation}
r_{j}(z)=%
\begin{cases}
\exp\bigg(-\dfrac{(z+\gamma_{j}\beta)^{2}}{1+c_{a}^{2}}\bigg) & \text{if
}z<0,\\
\exp\bigg(-\dfrac{2\beta z}{c_{a}^{2}+c_{s}^{2}}-\dfrac{\gamma_{j}^{2}%
\beta^{2}}{1+c_{a}^{2}}\bigg) & \text{if }z\geq0,
\end{cases}
\label{eq:RefAlpha0}%
\end{equation}
where $\beta$ is given by (\ref{eq:beta}). The function $r_{j}$ has an
exponential tail on the right and a Gaussian tail on the left. One can check
that the reference density given by (\ref{eq:reference}) and
(\ref{eq:RefAlpha0}) satisfies condition (\ref{eq:refcond}). In
(\ref{eq:RefAlpha0}), we set the shift term for $z<0$ to be $\gamma_{j}\beta$
according to the following observation. In the associated queue, $\beta$ is
the scaled mean number of idle servers and $\gamma_{j}$ is the fraction of
phase $j$ service load. In steady state, one can expect that $\tilde{Y}%
_{j}(t)$, the centered and scaled number of phase $j$ customers, is around
$-\gamma_{j}\beta$.

When $\alpha>0$ in (\ref{eq:model1}), the associated queue has abandonment. By
(\ref{eq:GMnGI}) and Conjecture~\ref{conj:TailModel1}, we choose
\begin{equation}
r_{j}(z)=%
\begin{cases}
\exp\bigg(-\dfrac{(z+\gamma_{j}\beta)^{2}}{1+c_{a}^{2}}\bigg) & \text{if
}z<0,\\
\exp\bigg(-\dfrac{\alpha(z+p_{j}\alpha^{-1}\mu\beta)^{2}}{\mu(c_{a}^{2}%
+c_{s}^{2})}+\dfrac{p_{j}^{2}\alpha^{-1}\mu\beta^{2}}{c_{a}^{2}+c_{s}^{2}%
}-\dfrac{\gamma_{j}^{2}\beta^{2}}{1+c_{a}^{2}}\bigg) & \text{if }z\geq0,
\end{cases}
\label{eq:Ref}%
\end{equation}
whose two tails are both Gaussian but have different decay rates. This
reference density also satisfies (\ref{eq:refcond}). In (\ref{eq:Ref}), the
shift term for $z\geq0$ is taken to be $p_{j}\mu\beta/\alpha$ because of the
observation below. When $\rho\geq1$, the throughput of the queue is nearly
$n\mu$. Let $q_{0}$ be the scaled queue length \textquotedblleft in
equilibrium\textquotedblright, i.e., the arrival and the departure rates of
the system are balanced when the queue length is around $n^{1/2}q_{0}$.
Because in this case the abandonment rate is $\alpha n^{1/2}q_{0}$, we must
have $\lambda=n\mu+\alpha n^{1/2}q_{0}$, or $q_{0}=-\mu\beta/\alpha$ by
(\ref{eq:beta}). Since the fraction of phase $j$ waiting customers is around
$p_{j}$, $\tilde{Y}_{j}(t)$ is around $-p_{j}\mu\beta/\alpha$ as the queue
reaches a steady state.

\subsection{Reference densities for model (\ref{eq:model2})}

\label{sec:densityHaz}

For the diffusion model (\ref{eq:model2}) that adopts the patience time hazard
rate scaling, the tail behavior of $X$ in steady state is left to future
research. In some cases, we may exploit the diffusion limit in
(\ref{eq:diffusionlim}) to facilitate the choice of a reference density for
the current model. The principle is again to ensure that the reference density
has a comparable or slower decay rate than the stationary density of $X$. For
that, we build an auxiliary queue that shares the same arrival process and
service times with the $GI/Ph/n+GI$ queue, but the auxiliary queue may have no
abandonment or have an exponential patience time distribution. Let $\hat{X}$
be the diffusion process in (\ref{eq:model1}) for the auxiliary queue. If
$\hat{X}$ has a slower decay rate than $X$, a reference density of $\hat{X}$
must be a reference density of $X$, too.

When $\rho<1$, the auxiliary queue is a $GI/Ph/n$ queue. It is intuitive that
the queue length decays faster in the $GI/Ph/n+GI$ queue than in the auxiliary
queue because the latter has no abandonment. As a consequence, $\hat{X}$ has a
slower decay rate than $X$ and the reference density given by
(\ref{eq:reference}) and (\ref{eq:RefAlpha0}) for $\hat{X}$ can be used for
the current model.

When $\rho>1$, the auxiliary queue is a $GI/Ph/n+M$ queue. Let $\alpha>0$ be
the rate of the exponential patience time distribution, which is to be
determined in order for $\hat{X}$ to have an appropriate decay rate. For that,
we need investigate the abandonment process of the $GI/Ph/n+GI$ queue.

Assume that the hazard rate function $h$ is $m$ times continuously
differentiable in a neighborhood of zero for some nonnegative integer $m$, and
among $\ell=0,\ldots,m$, there is at least one $h^{(\ell)}(0)\neq0$. We follow
the convention that $h^{(0)}(0)=h(0)$. Let $\ell_{0}$ be the smallest
nonnegative integer such that $h^{(\ell_{0})}(0)\neq0$. For $z>0$ in a small
neighborhood of zero, the $\ell_{0}$th degree Taylor's approximation of $h$ is%
\begin{equation}
h(z)\approx\frac{h^{(\ell_{0})}(0)z^{\ell_{0}}}{\ell_{0}!}, \label{eq:Taylor}%
\end{equation}
which, along with (\ref{eq:Q}) and (\ref{eq:hazaband}), implies that the
scaled abandonment process can be approximated by%
\[
\tilde{G}(t)\approx\frac{n^{\ell_{0}/2}h^{(\ell_{0})}(0)}{\lambda^{\ell_{0}%
}(\ell_{0}+1)!}\int_{0}^{t}(\tilde{Q}(s))^{\ell_{0}+1}\,\mathrm{d}s.
\]
This approximation implies that the abandonment process depends on the hazard
rate function primarily through $h^{(\ell_{0})}(0)$, the nonzero derivative at
the origin with the lowest order. It also implies that the scaled abandonment
rate at time $t$ is approximately%
\begin{equation}
\int_{0}^{\tilde{Q}(t)}h\Big(\frac{\sqrt{n}u}{\lambda}\Big)\,\mathrm{d}%
u\approx\frac{n^{\ell_{0}/2}h^{(\ell_{0})}(0)}{\lambda^{\ell_{0}}(\ell
_{0}+1)!}(\tilde{Q}(t))^{\ell_{0}+1}. \label{eq:abrate}%
\end{equation}

In the hazard rate scaling, the scaled queue length in equilibrium $q_{0}$
satisfies%
\begin{equation}
\lambda=n\mu+\sqrt{n}\int_{0}^{q_{0}}h\Big(\frac{\sqrt{n}u}{\lambda
}\Big)\,\mathrm{d}u. \label{eq:balance}%
\end{equation}
If (\ref{eq:abrate}) holds, it turns out to be%
\[
\lambda\approx n\mu+\frac{n^{(\ell_{0}+1)/2}h^{(\ell_{0})}(0)}{\lambda
^{\ell_{0}}(\ell_{0}+1)!}q_{0}^{\ell_{0}+1},
\]
which gives us%
\begin{equation}
q_{0}\approx\frac{1}{\sqrt{n}}\bigg(\frac{\lambda^{\ell_{0}}(\ell
_{0}+1)!(\lambda-n\mu)}{h^{(\ell_{0})}(0)}\bigg)^{1/(\ell_{0}+1)}.
\label{eq:q0}%
\end{equation}
The scaled queue length process fluctuates around this equilibrium length.
Correspondingly, the instantaneous abandonment rate changes around an
equilibrium level, too. This observation motivates us to take%
\begin{equation}
\alpha=\frac{n^{\ell_{0}/2}h^{(\ell_{0})}(0)}{\lambda^{\ell_{0}}(\ell_{0}%
+1)!}q_{0}^{\ell_{0}} \label{eq:alpha}%
\end{equation}
for the auxiliary $GI/Ph/n+M$ queue. With this setting, the original queue and
the auxiliary queue have comparable abandonment rates when the scaled queue
length is close to $q_{0}$. For any $q_{1}>q_{0}$, when the scaled queue
length is $q_{1}$ in both queues, the abandonment rate in the auxiliary queue
is lower because%
\[
\alpha q_{1}<\frac{n^{\ell_{0}/2}h^{(\ell_{0})}(0)}{\lambda^{\ell_{0}}%
(\ell_{0}+1)!}q_{1}^{\ell_{0}+1}.
\]
Hence, when the queue length is longer than $q_{0}$, it decays slower in the
auxiliary queue than in the original queue. Consequently, the decay rate of
$\hat{X}$ is slower than that of $X$ and the reference density of $\hat{X}$
can work for this diffusion model.

The above discussion suggests a product reference density in
(\ref{eq:reference}) with%
\begin{equation}
r_{j}(z)=%
\begin{cases}
\exp\bigg(-\dfrac{(z+\gamma_{j}\beta)^{2}}{1+c_{a}^{2}}\bigg) & \text{if
}z<0,\\
\exp\bigg(-\dfrac{\alpha(z-p_{j}q_{0})^{2}}{\mu(c_{a}^{2}+c_{s}^{2})}%
+\dfrac{\alpha p_{j}^{2}q_{0}^{2}}{\mu(c_{a}^{2}+c_{s}^{2})}-\dfrac{\gamma
_{j}^{2}\beta^{2}}{1+c_{a}^{2}}\bigg) & \text{if }z\geq0,
\end{cases}
\label{eq:refhaz}%
\end{equation}
where $q_{0}$ follows (\ref{eq:q0}) and $\alpha$ follows (\ref{eq:alpha}).

The above reference density fails when $\rho=1$ and $\ell_{0}>0$, because
$q_{0}=0$ by (\ref{eq:q0}) and thus $\alpha$ is zero in \eqref{eq:alpha}. In
this case, we can still choose a reference density by (\ref{eq:reference}) and
(\ref{eq:refhaz}) but using a traffic intensity $\rho$ that is slightly larger
than one. Because the tail of the queue length becomes heavier as $\rho$
increases, a reference density for model (\ref{eq:model2}) with $\rho>1$ must
have a comparable or slower decay rate than the stationary density of the
model with $\rho=1$.

The reference density in (\ref{eq:reference}) and (\ref{eq:refhaz}) that
exploits the lowest-order nonzero derivative at the origin may fail when the
hazard rate function has a rapid change near the origin. In this case, the
Taylor's approximation in (\ref{eq:Taylor}) may not be satisfactory when the
queue length is not short enough. Such an example is discussed in
Section~\ref{sec:Example4}. Choosing a reference density for that is more
flexible. In addition, the above procedure cannot choose a reference density
when all $h^{(\ell)}(0)$'s are zero, i.e., the hazard rate function is zero in
a neighborhood of the origin. This topic will be explored in future research.

\subsection{Truncation hypercube}

Once the reference density has been determined, we can choose the truncation
hypercube $K$ in (\ref{eq:K}) by the procedure below. First, pick a small
number $\varepsilon_{0}>0$. Then, choose a hypercube $K$ such that%
\begin{equation}
\int_{\mathbb{R}^{d}\setminus K}r(x)\,\mathrm{d}x<\varepsilon_{0}.
\label{eq:eps0}%
\end{equation}
When $\varepsilon_{0}$ is small enough, the influence of the reference density
outside $K$ is negligible in computing the stationary density.

\section{Numerical examples}

\label{sec:Numerical}

Several numerical examples are presented in this section. In each example, we
compute the stationary distribution of the number of customers in a
many-server queue using a diffusion model and the proposed algorithm. We
assume that the customer arrivals follow a homogeneous Poisson process and the
service times follow a two-phase hyperexponential distribution with mean one,
i.e., the system is an $M/H_{2}/n+GI$ queue with $c_{a}^{2}=1$ and $\mu=1$. In
such a queue, there are two types of customers. One type has a shorter mean
service time than the other, and the service times of either type are iid
following an exponential distribution. We approximate this queue by a
two-dimensional diffusion process $X$. When the patience time distribution is
exponential, both (\ref{eq:model1}) and (\ref{eq:model2}) yield the same
diffusion process. When the patience time distribution is non-exponential, we
use model (\ref{eq:model2}) as it is more accurate. The results computed using
the diffusion models are compared with the ones obtained either by the
matrix-analytic method or by simulation. Please refer to \cite{Neuts81} and
\cite{LatoucheRamaswami99} for the implementation of the matrix-analytic
method. All simulation results are obtained by averaging twenty runs and in
each run, the queue is simulated for one hundred thousand time units.

In the proposed algorithm, all numerical integration is implemented using a
Gauss-Legendre quadrature rule. See \cite{Kress98} for more details. When
computing $A_{i\ell}$ or $v_{i}$ in (\ref{eq:coeff}), the integrand is
evaluated at eight points in each dimension. In the numerical examples, the
tail probability
\begin{equation}
\mathbb{P}[X_{1}(\infty)+X_{2}(\infty)>z]=\int_{\{x\in\mathbb{R}^{2}%
:x_{1}+x_{2}>z\}}g(x)\,\mathrm{d}x\qquad\text{for some }z\in\mathbb{R}
\label{eq:Uz}%
\end{equation}
is also computed, where $X(\infty)=(X_{1}(\infty),X_{2}(\infty))^{\prime}$ is
the two-dimensional random vector having probability density $g$. The integral
in (\ref{eq:Uz}) is computed by adding up the integrals over the finite
elements that intersect with the set $\{x\in\mathbb{R}^{2}:x_{1}+x_{2}>z\}$,
and the integral over each finite element is again computed using a
Gaussian-Legendre quadrature formula. Because the indicator function has jumps
inside certain finite elements, we use sixty-four points in each dimension
when evaluating the integrand over each finite element.

\subsection{Example 1: an $M/H_{2}/n+M$ queue}

\label{sec:Example1}

\begin{figure}[t]
\centering
\subfloat[$\rho=1.141$ and $n=50$] {\label{fig:MH2M50}
\includegraphics[width=2.45in]{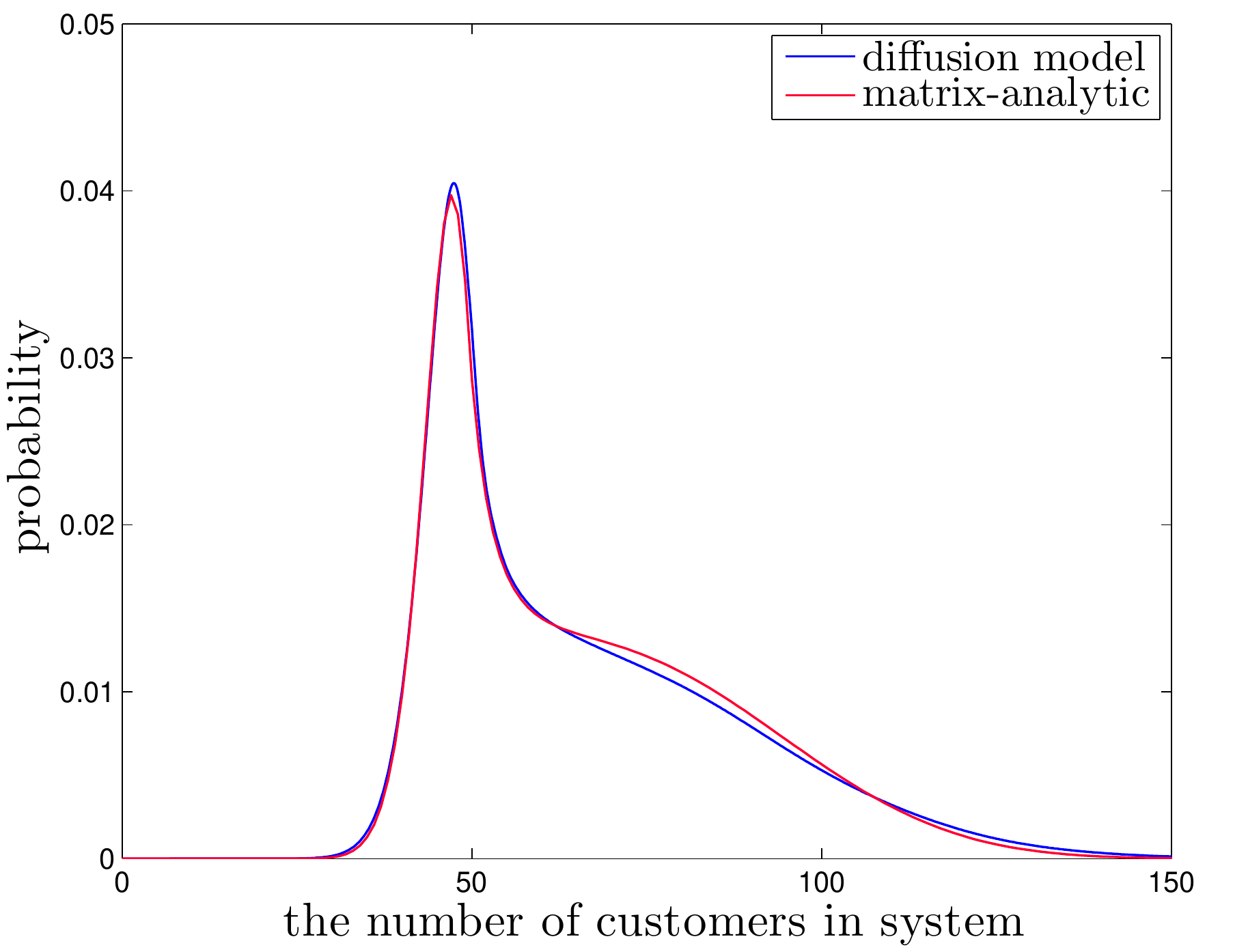}}
\subfloat[$\rho=1.045$ and $n=500$] {\label{fig:MH2M500}
\includegraphics[width=2.45in]{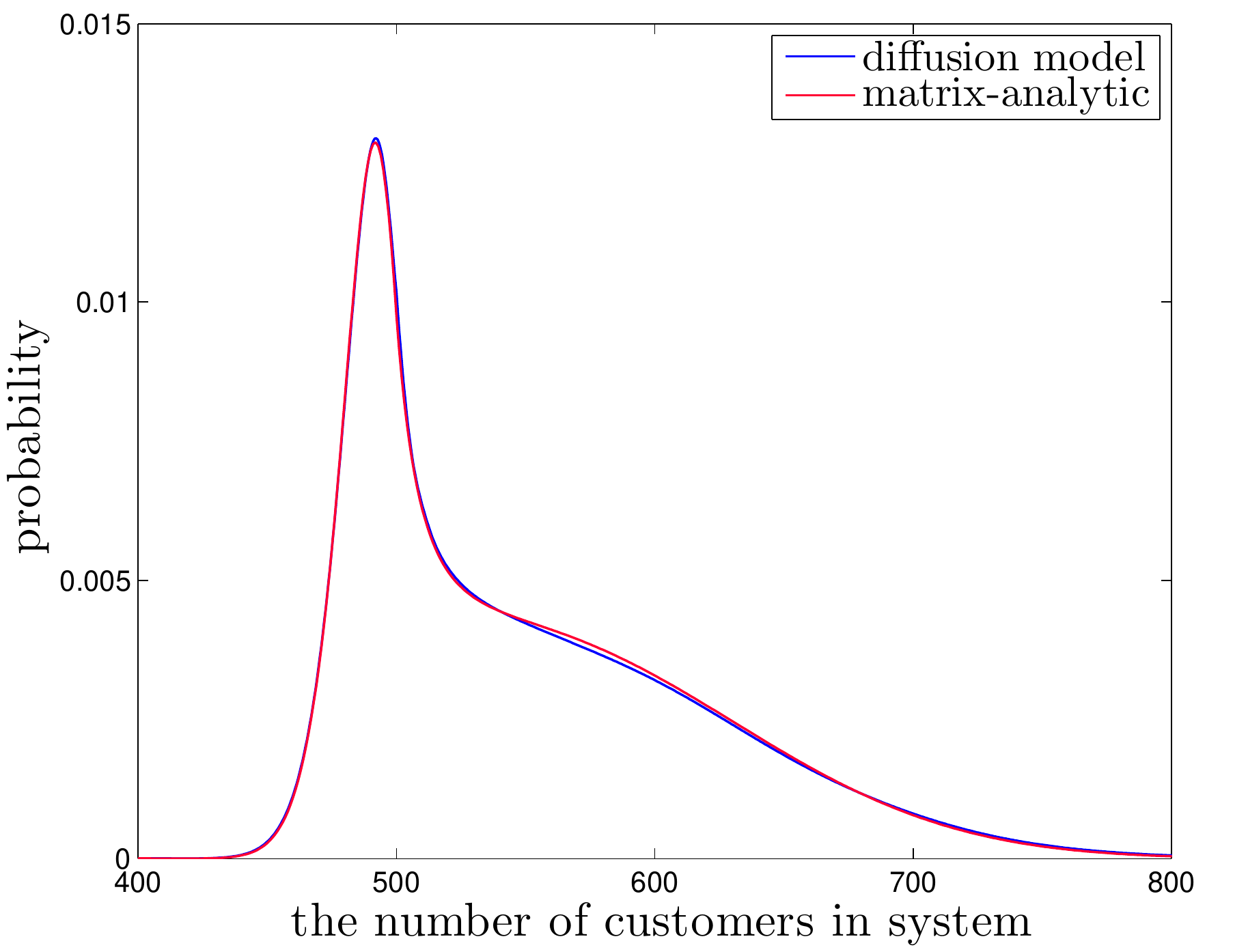}}\caption{The stationary
distribution of the customer number in the $M/H_{2}/n+M$ queue.}%
\label{fig:MH2nM}%
\end{figure}

\begin{table}[t]
\centering
\subfloat[$\rho=1.141$ and $n=50$]{\label{tab:MH250M}
\begin{tabular}
[c]{l|ll}
& Model (\ref{eq:model1}) & Matrix-analytic\\ \hline
Mean queue length & $17.27$ & $17.16$ \\
Abandonment fraction & $0.1512$ & $0.1503$ \\
$\mathbb{P}[N(\infty)>45]$ & $0.8675$ & $0.8523$ \\
$\mathbb{P}[N(\infty)>50]$ & $0.6785$ & $0.6726$ \\
$\mathbb{P}[N(\infty)>100]$ & $0.08700$ & $0.07436$ \\
$\mathbb{P}[N(\infty)>130]$ & $0.008662$ & $0.003299$
\end{tabular}}
\par
\subfloat[$\rho=1.045$ and $n=500$]{\label{tab:MH2500M}
\begin{tabular}
[c]{l|ll}
& Model (\ref{eq:model1}) & Matrix-analytic\\ \hline
Mean queue length & $54.17$ & $54.05$ \\
Abandonment fraction & $0.05181$ & $0.05173$ \\
$\mathbb{P}[N(\infty)>470]$ & $0.9701$ & $0.9694$ \\
$\mathbb{P}[N(\infty)>500]$ & $0.6838$ & $0.6818$ \\
$\mathbb{P}[N(\infty)>600]$ & $0.2244$ & $0.2229$ \\
$\mathbb{P}[N(\infty)>750]$ & $0.008233$ & $0.006395$
\end{tabular}}\caption{Performance measures of the $M/H_{2}/n+M$ queue.}%
\label{tab:MH2nM}%
\end{table}

\begin{figure}[t]
\centering
\includegraphics[width=2.45in]{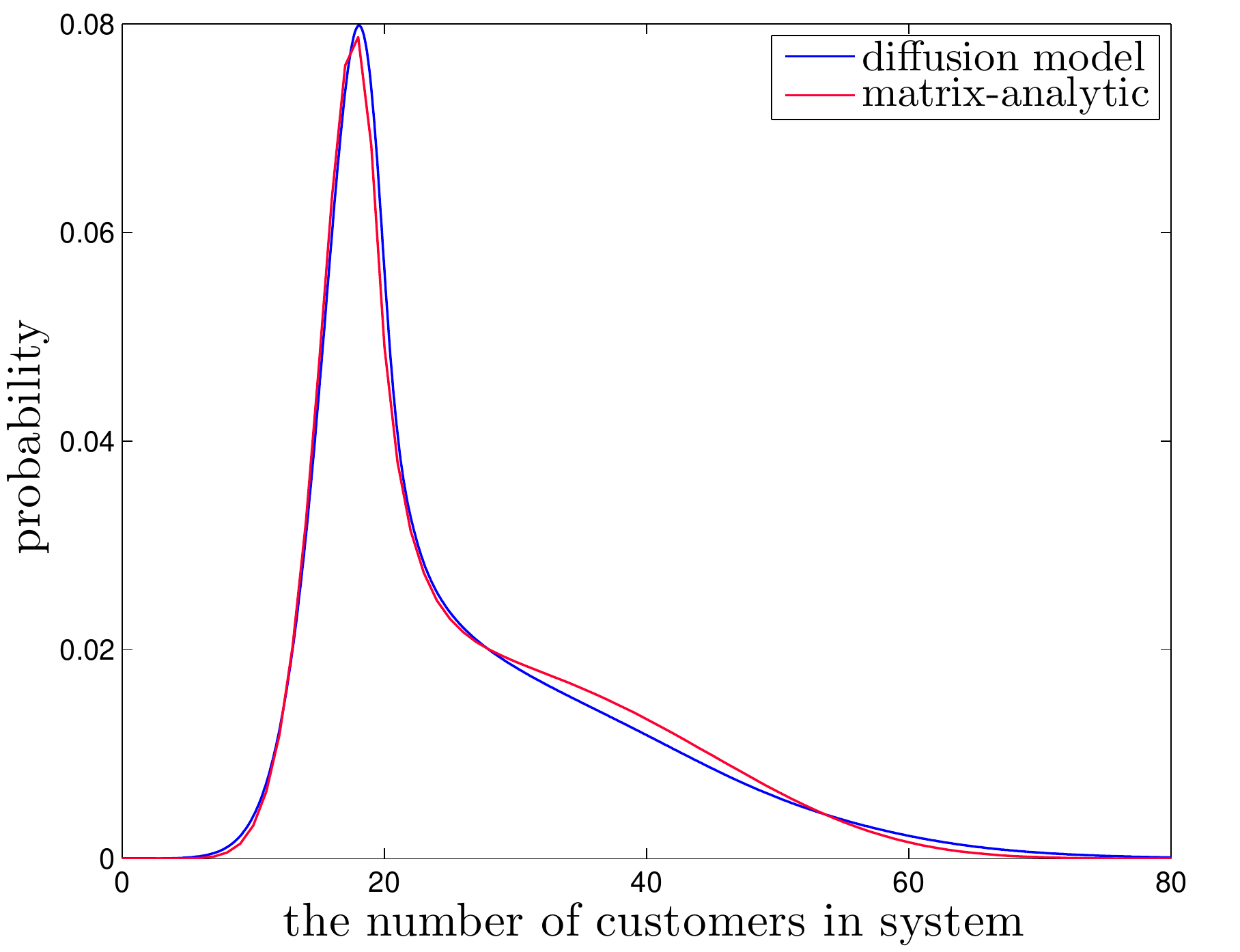}\caption{The stationary
distribution of the customer number in the $M/H_{2}/n+M$ queue, with
$\rho=1.112$ and $n=20$.}%
\label{fig:MH220M}%
\end{figure}

Consider an $M/H_{2}/n+M$ queue that has an exponential patience time
distribution. We are interested in such a queue because its customer-count
process $N=\{N(t):t\geq0\}$ is a quasi-birth-death process, where $N(t)$ is
the number of customers in system at time $t$. The stationary distribution of
that can be computed by the matrix-analytic method.

In this example, we take $\alpha=0.5$ for the rate of the exponential patience
time distribution and
\[
p=(0.9351,0.0649)^{\prime}\qquad\text{and}\qquad\nu=(9.354,0.072)^{\prime}%
\]
for the hyperexponential service time distribution. The mean service time of
the second-type customers is more than one hundred times longer than that of
the first type. Although over ninety percent of customers are of the first
type, the fraction of its workload is merely ten percent, i.e., $\gamma
=(0.1,0.9)^{\prime}$. Such a distribution has a large squared coefficient of
variation $c_{s}^{2}=24$.

The queue is approximated by the two-dimensional piecewise OU process $X$ in
(\ref{eq:model1}). Because the service time distribution is hyperexponential,
$P$ is a zero matrix and thus $R=\diag(\nu)$. By (\ref{eq:drift1}) and
(\ref{eq:covariance}), the drift coefficient of $X$ is
\begin{equation}
b(x)=%
\begin{pmatrix}
-p_{1}\mu\beta-\nu_{1}(x_{1}-p_{1}(x_{1}+x_{2})^{+})-p_{1}\alpha(x_{1}%
+x_{2})^{+}\\
-p_{2}\mu\beta-\nu_{2}(x_{2}-p_{2}(x_{1}+x_{2})^{+})-p_{2}\alpha(x_{1}%
+x_{2})^{+}%
\end{pmatrix}
\label{eq:b2}%
\end{equation}
and the covariance matrix of the diffusion coefficient is
\begin{equation}
\Sigma(x)=%
\begin{pmatrix}
p_{1}\mu(\rho+(\rho\wedge1)) & 0\\
0 & p_{2}\mu(\rho+(\rho\wedge1))
\end{pmatrix}
\label{eq:sig2}%
\end{equation}
for all $x\in\mathbb{R}^{2}$.

Three scenarios are considered in this example, in all of which the queue is
overloaded. In the first two scenarios, there are $n=50$ and $500$ servers,
respectively. The arrival rates are $\lambda=57.071$ and $522.36$, or
equivalently, $\rho=1.141$ and $1.045$. By (\ref{eq:beta}), $\beta=-1$ in both
scenarios. The third scenario, with $n=20$ servers, will be presented shortly.

To compute the stationary distribution of $X$, we use a product reference
density given by (\ref{eq:reference}) and (\ref{eq:Ref}). To generate basis
functions by the finite element method, we set the truncation rectangle
$K=[-7,32]\times\lbrack-7,32]$, which is obtained by (\ref{eq:eps0}) with
$\varepsilon_{0}=10^{-7}$, and use a lattice mesh in which all finite elements
are $0.5\times0.5$ squares.

Once the stationary density of $X$ is obtained, one can approximately produce
the distribution of $N(\infty)$, the stationary number of customers in system.
Note that the probability density of $X_{1}(\infty)+X_{2}(\infty)$ is given by%
\[
g_{N}(z)=\int_{-\infty}^{+\infty}g(x_{1},z-x_{1})\,\mathrm{d}x_{1}%
\qquad\text{for }z\in\mathbb{R}\text{.}%
\]
The distribution of $N(\infty)$\ can be approximated by%
\[
\mathbb{P}[N(\infty)=i]\approx\frac{1}{\sqrt{n}}g_{N}\Big(\frac{i-n}{\sqrt{n}%
}\Big)\qquad\text{for }i=0,1,\ldots
\]
For the first two scenarios, the distributions of $N(\infty)$ obtained by the
diffusion model are illustrated in Figure \ref{fig:MH2nM}. In the same figure,
the stationary distributions computed by the matrix-analytic method are
plotted, too. We see good agreement in Figure \ref{fig:MH2nM}. Comparing the
two scenarios, we also find out that the diffusion model in (\ref{eq:model1})
is more accurate when the number of servers $n$ is larger. This observation is
consistent with the many-server limit theorem for $G/Ph/n+GI$ queues in
\cite{DaiHeTezcan10}.

The matrix-analytic method can be used because in this queue, the
three-dimensional process $\{(Q(t),Z_{1}(t),Z_{2}(t)):t\geq0\}$ forms a
continuous-time Markov chain and the customer-count process $N$ is a
quasi-birth-death process. Clearly, $N(t)=Q(t)+Z_{1}(t)+Z_{2}(t)$. At time
$t$, $N$ is said to be at level $\ell$ if $N(t)=\ell$. In this example, level
$\ell$ consists of $\ell+1$ states if $\ell\leq n$ and it contains $n+1$
states if $\ell>n$. In the matrix-analytic method, the transition rate
matrices between adjacent levels are exploited to compute the stationary
distribution of $N$ iteratively. Each iteration requires $O(n^{3})$ arithmetic
operations. For this queue, the transition rate matrices at different levels
are different because the abandonment rate depends on the queue length. For
implementation purposes, we assume in the algorithm that at level $\ell
>\ell_{0}$ for some $\ell_{0}\gg n$, the abandonment rate at level $\ell$ is
$\alpha(\ell_{0}-n)$ rather than $\alpha(\ell-n)$. In other words, the
transition rate matrices at level $\ell$ are invariant with respect to $\ell
$\ when $\ell>\ell_{0}$. We take $\ell_{0}=n+2000$ in all numerical examples.
The extra error caused by this modification is negligible, because in this
queue, the queue length is in the order of $O(n^{1/2})$ and the chance of the
customer number exceeding $\ell_{0}$ is extremely rare.

To investigate the diffusion model in (\ref{eq:model1}) quantitatively, we
list some steady-state performance measures in Table \ref{tab:MH2nM}. They
include the mean queue length, the fraction of abandoned customers, and the
probabilities that the number of customers exceeds certain levels. Using the
diffusion model,
\[
\text{the mean queue length}\approx\sqrt{n}\int_{\mathbb{R}^{2}}(x_{1}%
+x_{2})^{+}g(x)\,\mathrm{d}x
\]
and%
\[
\text{the mean number of idle servers}\approx\sqrt{n}\int_{\mathbb{R}^{2}%
}(x_{1}+x_{2})^{-}g(x)\,\mathrm{d}x.
\]
It follows from the latter approximation that%
\[
\text{the abandonment fraction}\approx1-\frac{\mu}{\lambda}\Big(n-\sqrt{n}%
\int_{\mathbb{R}^{2}}(x_{1}+x_{2})^{-}g(x)\,\mathrm{d}x\Big).
\]
In the table, the tail probability $\mathbb{P}[N(\infty)>\ell]$ is
approximated by%
\[
\mathbb{P}[N(\infty)>\ell]\approx\mathbb{P}\Big[X_{1}(\infty)+X_{2}%
(\infty)>\frac{1}{\sqrt{n}}(\ell-n)\Big]\qquad\text{for }\ell=0,1,\ldots
\]
and $\mathbb{P}[X_{1}(\infty)+X_{2}(\infty)>(\ell-n)/\sqrt{n}]$ is computed
via (\ref{eq:Uz}). In both scenarios, the diffusion model produces
satisfactory numerical estimates.

The computational complexity of the proposed algorithm, whether in computation
time or in memory space, does not change with the number of servers $n$. In
contrast, the matrix-analytic method becomes computationally expensive when
$n$ is large. In particular, the memory usage becomes a serious constraint
when a huge number of iterations are required. For the $n=500$ scenario in
this example, it took around one hour to finish the matrix-analytic
computation and the peak memory usage is nearly five gigabytes. Using the
diffusion model and the proposed algorithm, it took less than one minute and
the peak memory usage is less than two hundred megabytes on the same computer.
See Section~\ref{sec:complexity} for more discussion on the computational complexity.

Although the diffusion model is motivated and derived from the theory of
many-server queues, it is still relevant for a queue with a modest number of
servers. In the third scenario, there are $n=20$ servers and the arrival rate
is $\lambda=22.24$. Thus, $\rho=1.112$ and $\beta=-0.5$. In the proposed
algorithm, we keep the same truncation rectangle and lattice mesh as in the
previous two scenarios, and the reference density is again from
(\ref{eq:reference}) and (\ref{eq:Ref}). As illustrated in Figure
\ref{fig:MH220M}, the diffusion model can still capture the exact stationary
distribution for a queue with as few as twenty servers.

\subsection{Example 2: an $M/H_{2}/n$ queue}

\label{sec:Example2}

\begin{figure}[t]
\centering
\subfloat[$\rho=0.8586$ and $n=50$] {\label{fig:MH250}
\includegraphics[width=2.45in]{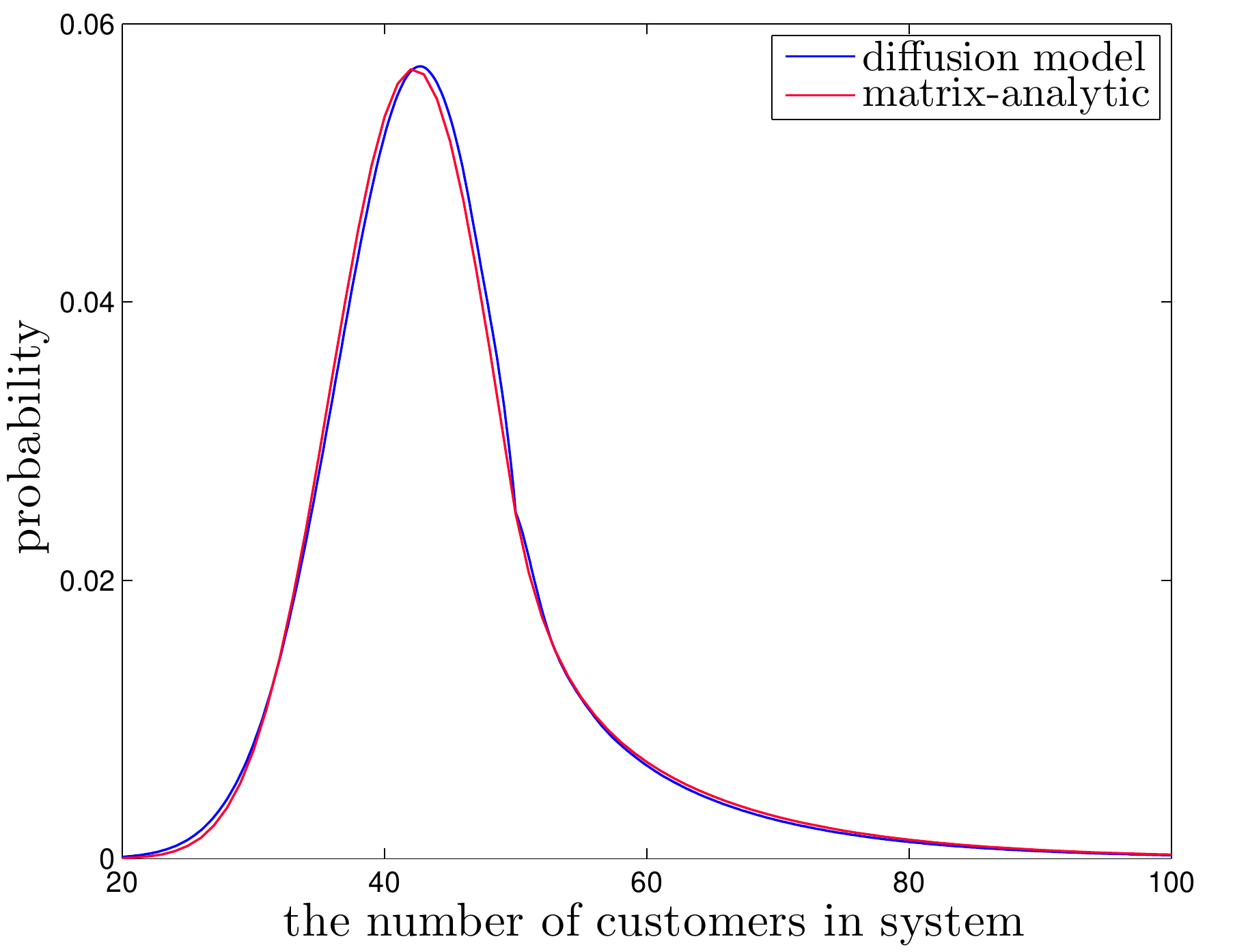}}
\subfloat[$\rho=0.9553$ and $n=500$] {\label{fig:MH2500}
\includegraphics[width=2.45in]{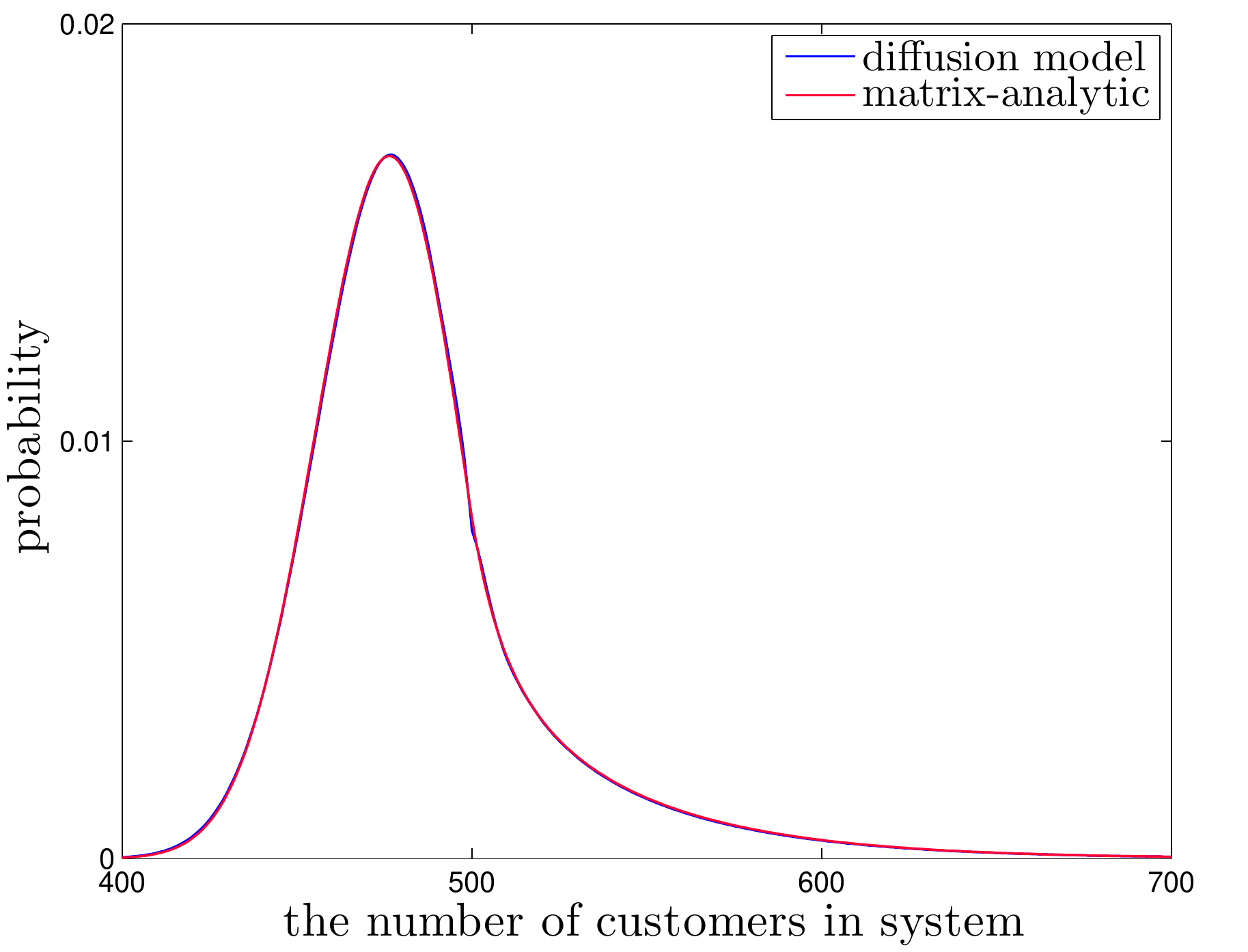}}\caption{The stationary distribution
of the customer number in the $M/H_{2}/n $ queue.}%
\label{fig:MH2n}%
\end{figure}

\begin{figure}[t]
\centering
\includegraphics[width=2.45in]{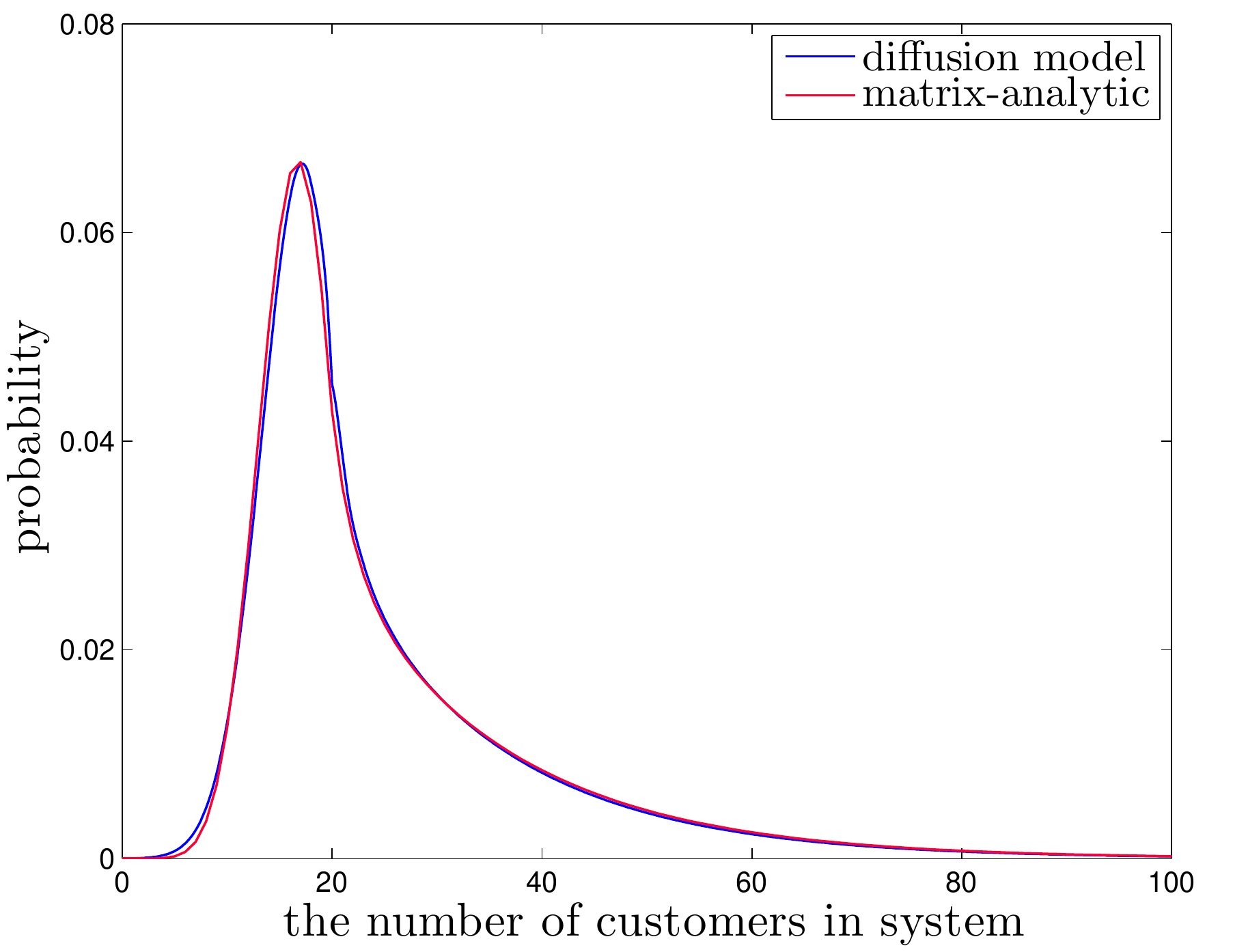}\caption{The stationary
distribution of the customer number in the $M/H_{2}/n $ queue, with
$\rho=0.8882$ and $n=20$.}%
\label{fig:MH220}%
\end{figure}

\begin{table}[t]
\centering
\subfloat[$\rho=0.8586$ and $n=50$]{\label{tab:MH250}
\begin{tabular}
[c]{l|ll}
& Model (\ref{eq:model1}) & Matrix-analytic \\ \hline
Mean queue length & $2.267$ & $2.419$\\
$\mathbb{P}[N(\infty)>40]$ & $0.6908$ & $0.6578$ \\
$\mathbb{P}[N(\infty)>50]$ & $0.2072$ & $0.2012$ \\
$\mathbb{P}[N(\infty)>70]$ & $0.03395$ & $0.03655$ \\
$\mathbb{P}[N(\infty)>100]$ & $0.003537$ & $0.003494$
\end{tabular}}
\par
\subfloat[$\rho=0.9553$ and $n=500$]{\label{tab:MH2500}
\begin{tabular}
[c]{l|ll}
& Model (\ref{eq:model1}) & Matrix-analytic \\ \hline
Mean queue length & $8.753$ & $8.800$\\
$\mathbb{P}[N(\infty)>450]$ & $0.9038$ & $0.9005$ \\
$\mathbb{P}[N(\infty)>500]$ & $0.2285$ & $0.2263$ \\
$\mathbb{P}[N(\infty)>600]$ & $0.01910$ & $0.01908$ \\
$\mathbb{P}[N(\infty)>700]$ & $0.002241$ & $0.001903$
\end{tabular}}\caption{Performance measures of the $M/H_{2}/n$ queue.}%
\label{tab:MH2n}%
\end{table}

In this example, an $M/H_{2}/n$ queue without abandonment is considered. The
hyperexponential service time distribution has%
\[
p=(0.5915,0.4085)^{\prime}\qquad\text{and}\qquad\nu=(5.917,0.454)^{\prime
}\text{.}%
\]
Thus, $c_{s}^{2}=3$ and $\gamma=(0.1,0.9)^{\prime}$. Since there is no
abandonment, we must take $\rho<1$ in order for the system to reach a steady state.

The diffusion model in (\ref{eq:model1}) with $\alpha=0$ is used. The drift
and the diffusion coefficients of $X$ are given by (\ref{eq:b2}) and
(\ref{eq:sig2}). The first scenario has $n=50$ servers and the second scenario
has $n=500$ servers. The respective arrival rates are $\lambda=42.929$ and
$477.64$. Hence, $\rho=0.8586$ and $0.9553$, both yielding $\beta=1$. The
product reference density is given by (\ref{eq:reference}) and
(\ref{eq:RefAlpha0}). With $\varepsilon_{0}=10^{-7}$, the truncation rectangle
is set by (\ref{eq:eps0}) to be $K=[-7,35]\times\lbrack-7,35]$, which is
divided into $0.5\times0.5$ finite elements.

The stationary distribution of the number of customers in system is shown in
Figure~\ref{fig:MH2n}. In both scenarios, the diffusion model produces a good
approximation of the result by the matrix-analytic method. As in the previous
example, the diffusion model is more accurate when the system scale is larger.
Several performance measures in steady state are listed in Table
\ref{tab:MH2n}. As in Table~\ref{tab:MH2nM}, satisfactory agreement can be
found between the two approaches.

The third scenario has $n=20$ servers with arrival rate $\lambda=17.76$. Then,
$\rho=0.8882$ and $\beta=0.5$. With $\varepsilon_{0}=10^{-7}$, the truncation
rectangle is taken to be $K=[-7,79]\times\lbrack-7,79]$. The lattice mesh
consists of $0.5\times0.5$ finite elements. The distribution of $N(\infty)$ is
shown in Figure~\ref{fig:MH220}. For a queue without abandonment, the
diffusion model is still useful when the number of servers is modest.

\subsection{Example 3: an $M/H_{2}/n+E_{k}$ queue}

\label{sec:Example3}

\begin{table}[tbh]
\centering
\subfloat[$\rho=0.8586$ and $n=50$]{\label{tab:MH250EU}
\begin{tabular}
[c]{l|ll|ll}
& \multicolumn{2}{c}{$+E_{2}$} & \multicolumn{2}{|c}{$+E_{3}$}\\ \cline{2-5}
& Model (\ref{eq:model2}) & Simulation & Model (\ref{eq:model2}) & Simulation \\ \hline
Mean queue length & $0.9820$ & $1.061$ & $1.201$ & $1.302$\\
Abandonment fraction & $0.007974$ & $0.008592$ & $0.005629$ & $0.006115$\\
$\mathbb{P}[N(\infty)>35]$ & $0.8881$ & $0.8745$ & $0.8896$ & $0.8762$\\
$\mathbb{P}[N(\infty)>40]$ & $0.6755$ & $0.6399$ & $0.6798$ & $0.6448$\\
$\mathbb{P}[N(\infty)>50]$ & $0.1671$ & $0.1581$ & $0.1788$ & $0.1707$\\
$\mathbb{P}[N(\infty)>60]$ & $0.03238$ & $0.03353$ & $0.04420$ & $0.04584$
\end{tabular}}
\par
\subfloat[$\rho=0.9553$ and $n=500$]{\label{tab:MH2500EU}
\begin{tabular}
[c]{l|ll|ll}
& \multicolumn{2}{c}{$+E_{2}$} & \multicolumn{2}{|c}{$+E_{3}$}\\ \cline{2-5}
& Model (\ref{eq:model2}) & Simulation & Model (\ref{eq:model2}) & Simulation \\ \hline
Mean queue length & $4.960$ & $5.048$ & $6.455$ & $6.569$\\
Abandonment fraction & $0.001689$ & $0.001729$ & $0.0007611$ & $0.0007931$\\
$\mathbb{P}[N(\infty)>450]$ & $0.9003$ & $0.8964$ & $0.9022$ & $0.8984$\\
$\mathbb{P}[N(\infty)>480]$ & $0.4759$ & $0.4643$ & $0.4859$ & $0.4746$\\
$\mathbb{P}[N(\infty)>500]$ & $0.1995$ & $0.1966$ & $0.2151$ & $0.2124$\\
$\mathbb{P}[N(\infty)>550]$ & $0.02798$ & $0.02841$ & $0.04412$ & $0.04458$
\end{tabular}}\caption{Performance measures of the $M/H_{2}/n+E_{k}$ queue
with $\rho<1$.}%
\label{tab:MH2nEU}%
\end{table}

\begin{table}[tbh]
\centering
\subfloat[$\rho=1.141$ and $n=50$]{\label{tab:MH250EO}
\begin{tabular}
[c]{l|ll|ll}
& \multicolumn{2}{c}{$+E_{2}$} & \multicolumn{2}{|c}{$+E_{3}$}\\ \cline{2-5}
& Model (\ref{eq:model2}) & Simulation & Model (\ref{eq:model2}) & Simulation \\ \hline
Mean queue length & $15.03$ & $14.94$ & $19.44$ & $19.31$\\
Abandonment fraction & $0.1332$ & $0.1334$ & $0.1303$ & $0.1305$\\
$\mathbb{P}[N(\infty)>45]$ & $0.9568$ & $0.9490$  & $0.9704$ & $0.9645$ \\
$\mathbb{P}[N(\infty)>50]$ & $0.8780$ & $0.8648$  & $0.9169$ & $0.9066$\\
$\mathbb{P}[N(\infty)>70]$ & $0.3325$ & $0.3121$  & $0.5037$ & $0.4761$\\
$\mathbb{P}[N(\infty)>90]$ & $0.008153$ & $0.009354$ & $0.03033$ & $0.03422$
\end{tabular}}
\par
\subfloat[$\rho=1.045$ and $n=500$]{\label{tab:MH2500EO}
\begin{tabular}
[c]{l|ll|ll}
& \multicolumn{2}{c}{$+E_{2}$} & \multicolumn{2}{|c}{$+E_{3}$}\\ \cline{2-5}
& Model (\ref{eq:model2}) & Simulation & Model (\ref{eq:model2}) & Simulation \\ \hline
Mean queue length & $76.50$ & $76.20$ & $119.5$ & $119.1$\\
Abandonment fraction & $0.04438$ & $0.04437$ & $0.04340$ & $0.04337$\\
$\mathbb{P}[N(\infty)>480]$ & $0.9857$ & $0.9846$ & $0.9946$ & $0.9940$\\
$\mathbb{P}[N(\infty)>500]$ & $0.9390$ & $0.9363$ & $0.9770$ & $0.9756$\\
$\mathbb{P}[N(\infty)>600]$ & $0.3115$ & $0.3051$ & $0.6733$ & $0.6645$\\
$\mathbb{P}[N(\infty)>700]$ & $0.0009757$ & $0.0009658$ & $0.04260$ & $0.04358$
\end{tabular}}\caption{Performance measures of the $M/H_{2}/n+E_{k}$ queue
with $\rho>1$.}%
\label{tab:MH2nEO}%
\end{table}

Consider an $M/H_{2}/n+E_{k}$ queue, where $k>1$ is a positive integer and
$+E_{k}$ signifies an Erlang-$k$ patience time distribution. In this queue,
each patience time is the sum of $k$ stages and the stages are iid having an
exponential distribution with mean $1/\theta$. When $k>1$, the probability
density at zero of an Erlang-$k$ distribution is zero. The diffusion model in
(\ref{eq:model1}) does not have a stationary distribution when the queue is
overloaded. Hence, we evaluate the diffusion model in (\ref{eq:model2}) that
exploits the patience time hazard rate scaling. In the following numerical
experiments, we take $k=2$ or $3$ for the Erlang-$k$ distribution and set
$\theta=k$, so the mean patience time is one unit time. The hyperexponential
service time distribution is taken to be identical to that in
Section~\ref{sec:Example2}.

The hazard rate function of the Erlang-$k$ distribution is%
\[
h(t)=\frac{\theta^{k}t^{k-1}}{(k-1)!\sum\limits_{\ell=0}^{k-1}\dfrac
{\theta^{\ell}t^{\ell}}{\ell!}}\qquad\text{for }t\geq0\text{.}%
\]
For the diffusion model (\ref{eq:model2}), it follows from (\ref{eq:drift2})
that the drift coefficient of $X$ is
\begin{equation}
b(x)=%
\begin{pmatrix}
-p_{1}\mu\beta-\nu_{1}(x_{1}-p_{1}(x_{1}+x_{2})^{+})-p_{1}\eta((x_{1}%
+x_{2})^{+})\\
-p_{2}\mu\beta-\nu_{2}(x_{2}-p_{2}(x_{1}+x_{2})^{+})-p_{2}\eta((x_{1}%
+x_{2})^{+})
\end{pmatrix}
\label{eq:b}%
\end{equation}
where
\[
\eta(z)=\int_{0}^{z}h\Big(\frac{\sqrt{n}u}{\lambda}\Big)\,\mathrm{d}u=\theta
z-\frac{\lambda}{\sqrt{n}}\log\bigg(\sum_{m=0}^{k-1}\frac{n^{m/2}\theta
^{m}z^{m}}{m!\lambda^{m}}\bigg)\qquad\text{for }z\geq0\text{.}%
\]

The first two scenarios has $n=50$ and $500$ servers, respectively. Their
respective arrival rates are $\lambda=42.929$ and $477.64$. Hence,
$\rho=0.8586$ and $0.9553$, both leading to $\beta=1$. In the proposed
algorithm, the reference density is chosen according to (\ref{eq:reference})
and (\ref{eq:RefAlpha0}). The truncation rectangle is taken to be
$K=[-7,35]\times\lbrack-7,35]$ and is divided into $0.5\times0.5$ finite
elements. Some performance estimates can be found in Table \ref{tab:MH2nEU}.

The third and fourth scenarios are for the case $\rho>1$. They have $n=50$ and
$500$ servers, and arrival rates $\lambda=57.071$ and $522.36$, respectively.
Then, $\rho=1.141$ and $1.045$, both having $\beta=-1$. For these two
scenarios, we adopt the reference density in (\ref{eq:reference}) and
(\ref{eq:refhaz}). When $k=2$, each patience time has two stages. The hazard
rate function of the patience time distribution has $h(0)=0$ and
$h^{(1)}(0)=\theta^{2}$, so $\ell_{0}=1$ in (\ref{eq:q0}) and (\ref{eq:alpha}%
). Because $\alpha$ in (\ref{eq:alpha}) depends on $n$, both the reference
density and the truncation rectangle change with $n$. With $\varepsilon
_{0}=10^{-7}$, the truncation rectangle is set to be $K=[-7,13]\times
\lbrack-7,13]$ for $n=50$ and to be $K=[-7,16]\times\lbrack-7,16]$ for
$n=500$. When $k=3$, a patience time consists of three stages. In this case,
$h(0)=h^{(1)}(0)=0$ and $h^{(2)}(0)=8\theta^{3}$, so $\ell_{0}=2$. We set
$K=[-7,11]\times\lbrack-7,11]$ for $n=50$ and $K=[-7,15]\times\lbrack-7,15]$
for $n=500$. All truncation rectangles are partitioned into $0.5\times0.5$
finite elements. The performance estimates are listed in Table
\ref{tab:MH2nEO}.

To evaluate the diffusion model (\ref{eq:model2}), we list corresponding
simulation estimates of the performance measures in both tables. As in the
previous examples, the diffusion model produces adequate performance approximations.

Theoretically, the matrix-analytic method can be used in this example as the
customer-count process $N$ is also a quasi-birth-death process. But it is
impractical because the computational complexity is too high. Consider the
case $k=2$. Let $V_{1}(t)$ and $V_{2}(t)$ be the respective numbers of waiting
customers whose patience times are in the first and in the second stage at
time $t$. For this $M/H_{2}/n+E_{2}$ queue, the four-dimensional process
$\{(V_{1}(t),V_{2}(t),Z_{1}(t),Z_{2}(t)):t\geq0\}$ is a continuous-time Markov
chain. At level $\ell$, there are $\ell+1$ states if $\ell\leq n$ and there
are $(n+1)(\ell-n+1)$ states if $\ell>n$. The number of states at level $\ell$
is formidable when $\ell$ is large. Even if we may truncate the state space
using the technique described in Section~\ref{sec:Example1}, the number of
states is still too large to apply the matrix-analytic method. In fact, we are
not aware of any other numerical methods other than simulation that can
produce approximations in Tables~\ref{tab:MH2nEU} and~\ref{tab:MH2nEO}.

\subsection{Example 4: an $M/H_{2}/n+H_{2}$ queue}

\label{sec:Example4}

\begin{table}[t]
\centering
\subfloat[$\rho=1.141$ and $n=50$]{\label{tab:MH250H2}
\begin{tabular}
[c]{l|lll}
& Model (\ref{eq:model1}) & Model (\ref{eq:model2}) & Simulation \\ \hline
Mean queue length & $0.4709$ & $4.869$ & $4.845$\\
Abandonment fraction & $0.1714$ & $0.1504$ & $0.1499$\\
$\mathbb{P}[N(\infty)>40]$ & $0.9578$ & $0.9749$ & $0.9728$\\
$\mathbb{P}[N(\infty)>50]$ & $0.3158$ & $0.6377$ & $0.6111$\\
$\mathbb{P}[N(\infty)>60]$ & $1.044\times10^{-7}$ & $0.1895$ & $0.1737$\\
$\mathbb{P}[N(\infty)>70]$ & $1.097\times10^{-11}$ & $0.02568$ & $0.02142$
\end{tabular}}
\par
\subfloat[$\rho=1.045$ and $n=500$]{\label{tab:MH2500H2}
\begin{tabular}
[c]{l|lll}
& Model (\ref{eq:model1}) & Model (\ref{eq:model2}) & Simulation \\ \hline
Mean queue length & $1.475$ & $6.359$ & $6.413$\\
Abandonment fraction & $0.05863$ & $0.05517$ & $0.05512$\\
$\mathbb{P}[N(\infty)>480]$ & $0.8663$ & $0.8929$ & $0.8881$\\
$\mathbb{P}[N(\infty)>500]$ & $0.3192$ & $0.4822$ & $0.4720$\\
$\mathbb{P}[N(\infty)>520]$ & $9.274\times10^{-5}$ & $0.1074$ & $0.1050$\\
$\mathbb{P}[N(\infty)>550]$ & $-4.488\times10^{-9}$ & $0.006616$ & $0.006248$
\end{tabular}}\caption{Performance measures of the $M/H_{2}/n+H_{2}$ queue.}%
\label{tab:MH2nH2}%
\end{table}

Let us consider an example in which the patience time hazard rate function
changes rapidly near the origin. As pointed out by \cite{ReedTezcan09}, the
performance of such a queue is sensitive to the patience time distribution in
a neighborhood of zero. A model that exploits the patience time density at
zero solely may not produce adequate performance estimates. In this example,
the patience times follow a two-phase hyperexponential distribution that has%
\[
\hat{p}=(0.9,0.1)^{\prime}\qquad\text{and}\qquad\hat{\nu}=(1,200)^{\prime
}\text{.}%
\]
In other words, there are two types of patience times. Ninety percent of
patience times are exponentially distributed with mean one and ten percent are
exponentially distributed with mean $0.005$. We take the same hyperexponential
service time distribution as in Sections~\ref{sec:Example2}
and~\ref{sec:Example3}.

The hazard rate function of the hyperexponential patience time distribution is%
\[
h(t)=\frac{\hat{p}_{1}\hat{\nu}_{1}\exp(-\hat{\nu}_{1}t)+\hat{p}_{2}\hat{\nu
}_{2}\exp(-\hat{\nu}_{2}t)}{\hat{p}_{1}\exp(-\hat{\nu}_{1}t)+\hat{p}_{2}%
\exp(-\hat{\nu}_{2}t)}\qquad\text{for }t\geq0\text{.}
\]
The drift coefficient of $X$ in (\ref{eq:model2}) is also given by
(\ref{eq:b}) where%
\begin{align*}
\eta(z) &  =\int_{0}^{z}h\Big(\frac{\sqrt{n}u}{\lambda}\Big)\,\mathrm{d}u\\
&  =-\frac{\lambda}{\sqrt{n}}\log\Big(\hat{p}_{1}\exp\Big(-\frac{\sqrt{n}
}{\lambda}\hat{\nu}_{1}z\Big)+\hat{p}_{2}\exp\Big(-\frac{\sqrt{n}}{\lambda
}\hat{\nu}_{2}z\Big)\Big)\qquad\text{for }z\geq0\text{.}%
\end{align*}
In this example, we have
\[
h(0)=\hat{p}_{1}\hat{\nu}_{1}+\hat{p}_{2}\hat{\nu}_{2}=20.9\qquad
\text{and\qquad}h^{(1)}(0)=-\hat{p}_{1}\hat{p}_{2}(\hat{\nu}_{1}-\hat{\nu}%
_{2})^{2}=-3564.1.
\]
Thus, $\ell_{0}=0$ and the hazard rate function has a steep slope near the
origin. Since the zeroth degree Taylor's approximation in (\ref{eq:Taylor})
may bring on too much error, the reference density exploiting the lowest-order
nonzero derivative at the origin could be erroneous.

To choose an appropriate reference density, an auxiliary queue is used again.
As in Section~\ref{sec:densityHaz}, the auxiliary queue is an $M/H_{2}/n+M$
queue that shares the same arrivals and service times with the $M/H_{2}%
/n+H_{2}$ queue. Let $\alpha>0$ be the rate of the exponential patience time
distribution. We take $\alpha=\hat{\nu}_{1}\wedge\hat{\nu}_{2}$ so that the
patience times in the auxiliary queue are all of the type with the longer
mean. If the queue lengths are equal, the abandonment rate in the auxiliary
queue must be lower than that in the original queue. Therefore, the queue
length decays slower in the former queue and the reference density for model
(\ref{eq:model1}) of the auxiliary queue should work. This observation leads
to a reference density that follows (\ref{eq:reference}) and (\ref{eq:refhaz}%
), but in this example, we take $\alpha=\hat{\nu}_{1}\wedge\hat{\nu}_{2}$ and
solve (\ref{eq:balance}) to find $q_{0}$.

Two scenarios with $n=50$ and $500$ servers are investigated. The respective
arrival rates are $\lambda=57.071$ and $522.36$. Thus, $\rho=1.141$ and
$1.045$ and both scenarios have $\beta=-1$. By solving (\ref{eq:balance}), we
have $q_{0}=0.165$ for the first scenario and $q_{0}=0.0059$ for the second
scenario. The reference density follows (\ref{eq:reference}) and
(\ref{eq:refhaz}) with $\alpha=\hat{\nu}_{1}=1$. With $\varepsilon_{0}%
=10^{-7}$, the truncation rectangle is $K=[-7,9]\times\lbrack-7,9]$,
partitioned into $0.5\times0.5$ finite elements. The performance estimates
obtained by the diffusion model (\ref{eq:model2}) are compared with the
simulation results in Table \ref{tab:MH2nH2}. The performance estimates are
still quite accurate.

We also put the performance estimates produced by the diffusion model
(\ref{eq:model1}) in this table. For this model, the reference density follows
(\ref{eq:reference}) and (\ref{eq:Ref}) with $\alpha=h(0)=20.9$. In the
proposed algorithm, the mesh for model (\ref{eq:model2}) is used again.
Because in this example, using the patience time density at zero solely cannot
capture the behavior of the abandonment process, model (\ref{eq:model1}) fails
to produce proper performance estimates.

\section{Implementation issues}

\label{sec:Implementation}

The proposed algorithm was implemented using the
\textrm{C\raise.3ex\hbox{\small++}} programming language. The package was
tested on both Linux and Windows platforms. In this section, we discuss
several important issues arising from the implementation. They are crucial for
using the algorithm to solve practical problems. To demonstrate these issues,
the second scenario with $n=500$ servers in Section~\ref{sec:Example1} is
investigated throughout this section. The diffusion model (\ref{eq:model1}) is
used to approximate the $M/H_{2}/n+M$ queue.

\subsection{Influence of the reference density}

\label{sec:InfluenceReference}

\begin{figure}[t]
\centering
\subfloat[$K={[-7,10]\times [-7,10]}$] {\label{fig:MH2500naive1}
\includegraphics[width=2.45in]{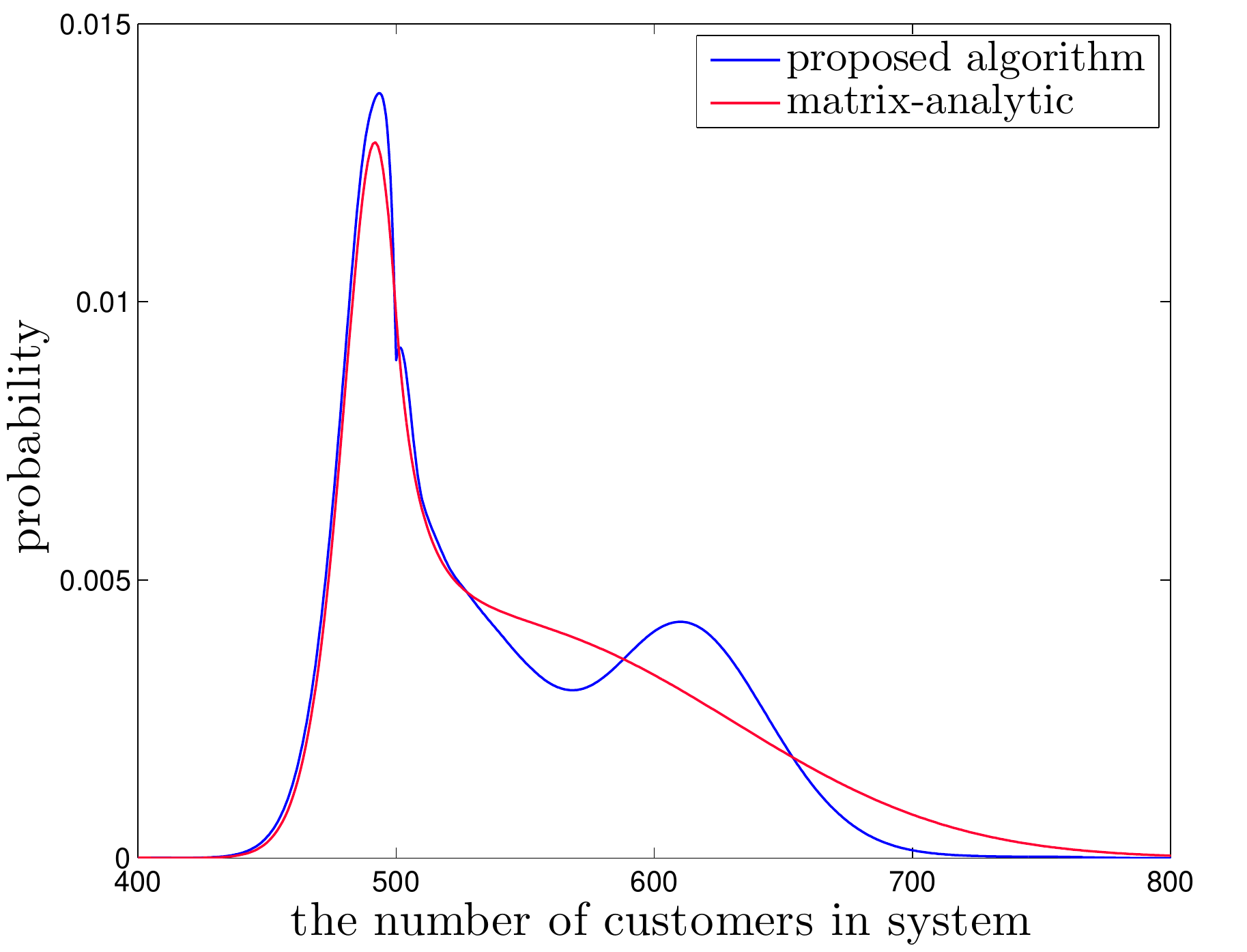}}
\subfloat[$K={[-7,32]\times [-7,32]}$] {\label{fig:MH2500naive2}
\includegraphics[width=2.45in]{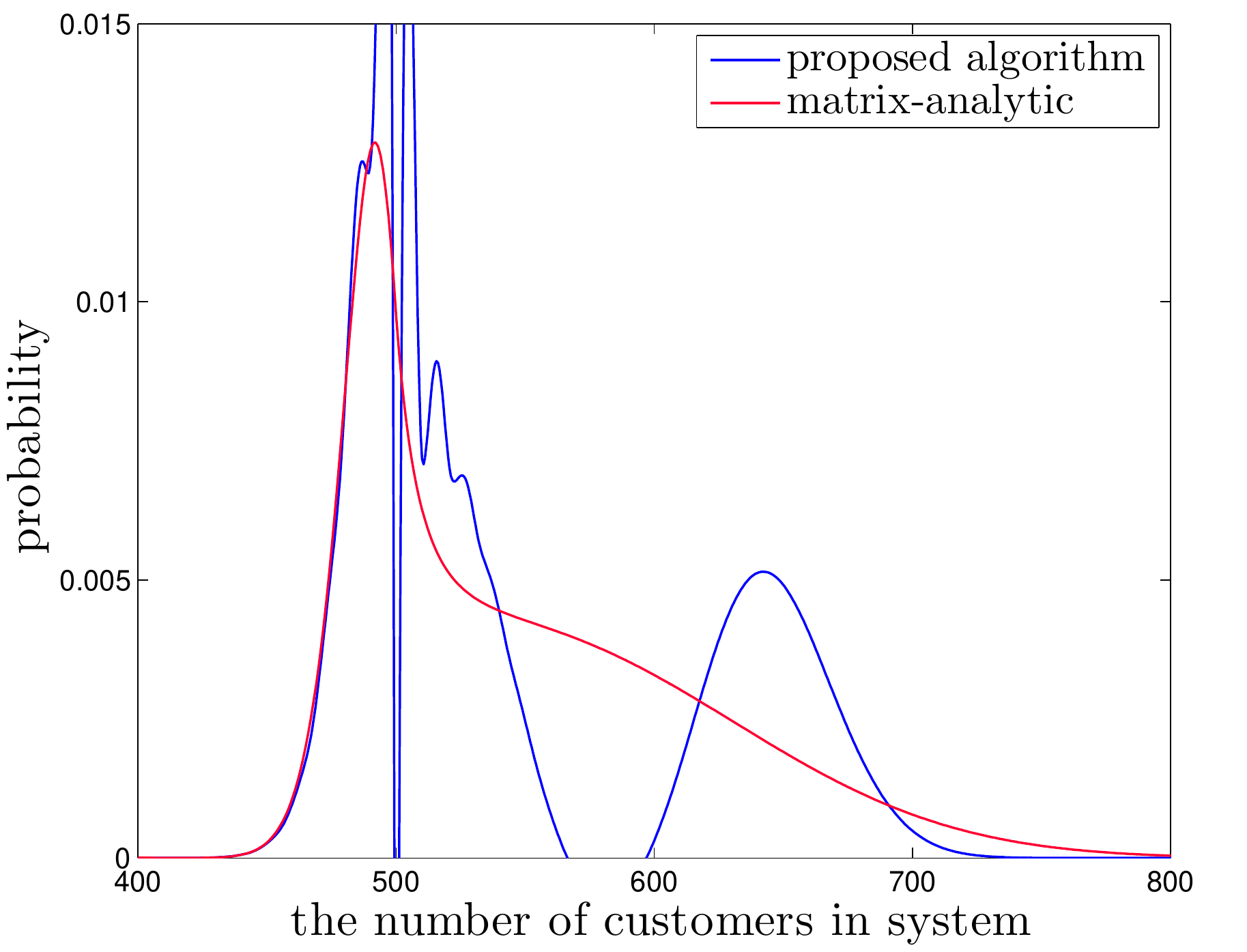}}\caption{The output of the
proposed algorithm with the ``naive'' reference density.}%
\label{fig:MH2nNaive}%
\end{figure}

The reference density plays a key role in the algorithm. If the function $r$
does not satisfy (\ref{eq:ref}), the sequence of spaces $\{H_{k}%
:k\in\mathbb{N}\}$ may not converge to $H$ in $L^{2}(\mathbb{R}^{d},r)$ and
the output of the algorithm may significantly deviate from the exact
stationary density. To demonstrate this issue, let us consider a
\textquotedblleft naive\textquotedblright\ reference density.

To produce a \textquotedblleft naive\textquotedblright\ reference density, we
consider a queue that has the same arrival process and patience time
distribution as the $M/H_{2}/n+M$ queue. This new queue has an exponential
service time distribution and its mean service time is equal to that of the
$M/H_{2}/n+M$ queue. For this $M/M/n+M$ queue, the diffusion model
(\ref{eq:model1}) is a one-dimensional piecewise OU process whose stationary
density is given by (\ref{eq:GMnGI}). The \textquotedblleft
naive\textquotedblright\ reference density is a product reference density in
(\ref{eq:reference}) with each $r_{j}$ being the stationary density in
(\ref{eq:GMnGI}). In other words, the \textquotedblleft
naive\textquotedblright\ reference density is obtained by pretending the
service time distribution to be exponential.

Let us apply the \textquotedblleft naive\textquotedblright\ reference density.
With $\varepsilon_{0}=10^{-7}$, the truncation rectangle is set to be
$K=[-7,10]\times\lbrack-7,10]$ and is partitioned into $0.5\times0.5$ finite
elements. As shown in Figure~\ref{fig:MH2500naive1}, the output of the
proposed algorithm noticeably deviates from the exact stationary distribution.
To further confirm that the \textquotedblleft naive\textquotedblright%
\ reference density cannot work, we next test the truncation rectangle
$K=[-7,32]\times\lbrack-7,32]$, which is used in Section~\ref{sec:Example1}
along with the proposed\ reference density. In this case, the matrix $A$ in
(\ref{eq:equation}) is close to singular and its condition number is
$3.52\times10^{190}$. Figure~\ref{fig:MH2500naive2} manifests the severe error
in the algorithm output.

Recall that in this example, the hyperexponential service time distribution
has $c_{s}^{2}=24$. Comparing (\ref{eq:GMnGI}) with (\ref{eq:Ref}), we can
tell that the decay rate of the \textquotedblleft naive\textquotedblright%
\ reference density is much larger than that of the proposed reference
density. If Conjecture~\ref{conj:TailModel1} is true, one can expect that the
\textquotedblleft naive\textquotedblright\ reference density decays much
faster than the stationary density and the second condition in (\ref{eq:ref})
may not hold. In this case, the ratio function $q$ is no longer in
$L^{2}(\mathbb{R}^{d},r)$ and consequently, the algorithm fails to produce any
adequate estimate of the ratio function.

\subsection{Mesh selection}

\label{sec:Mesh}

\begin{figure}[t]
\centering
\subfloat[$1.0\times 1.0$ finite elements] {\label{fig:FE1}
\includegraphics[width=2.45in]{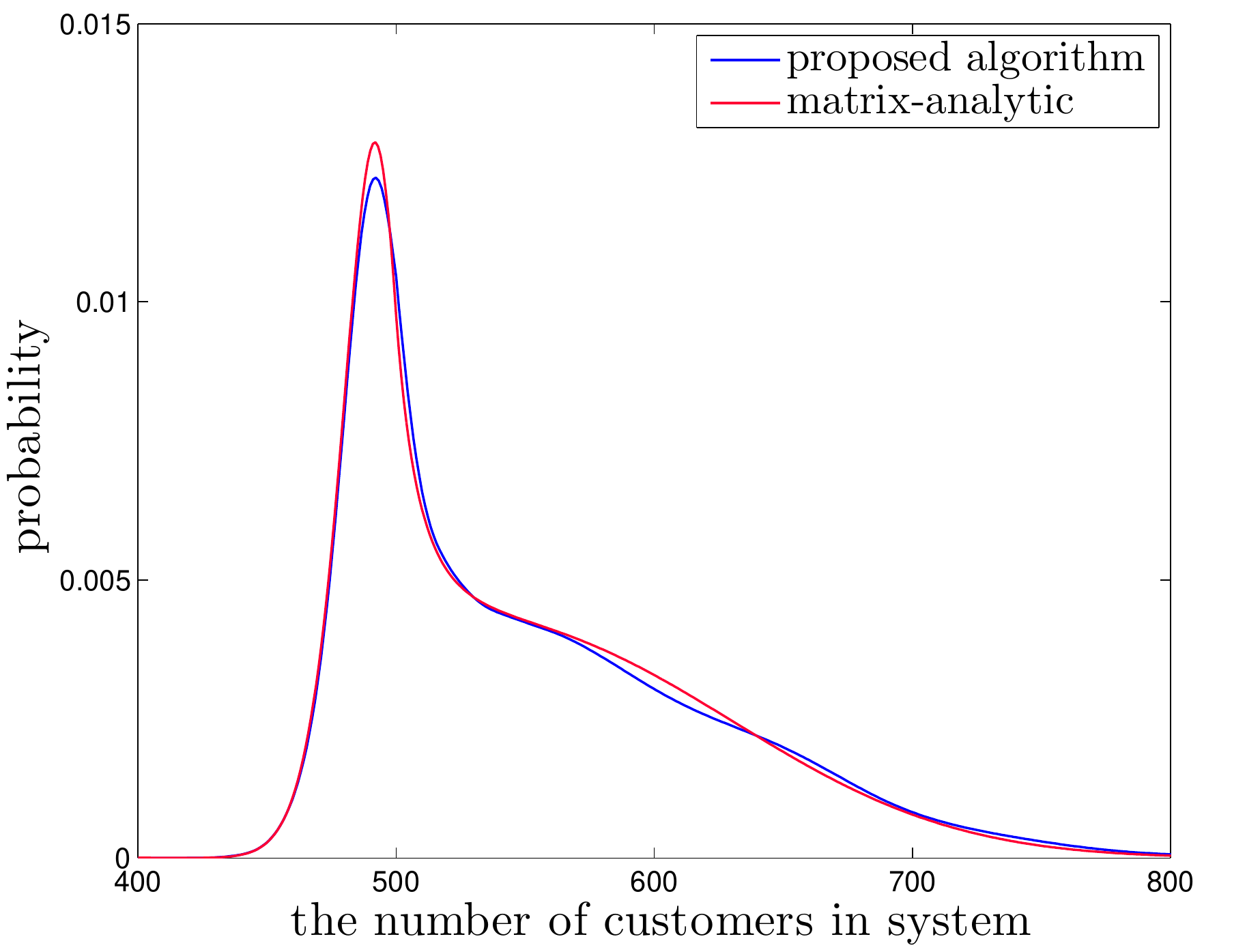}}
\subfloat[$0.25\times 0.25$ finite elements] {\label{fig:FE025}
\includegraphics[width=2.45in]{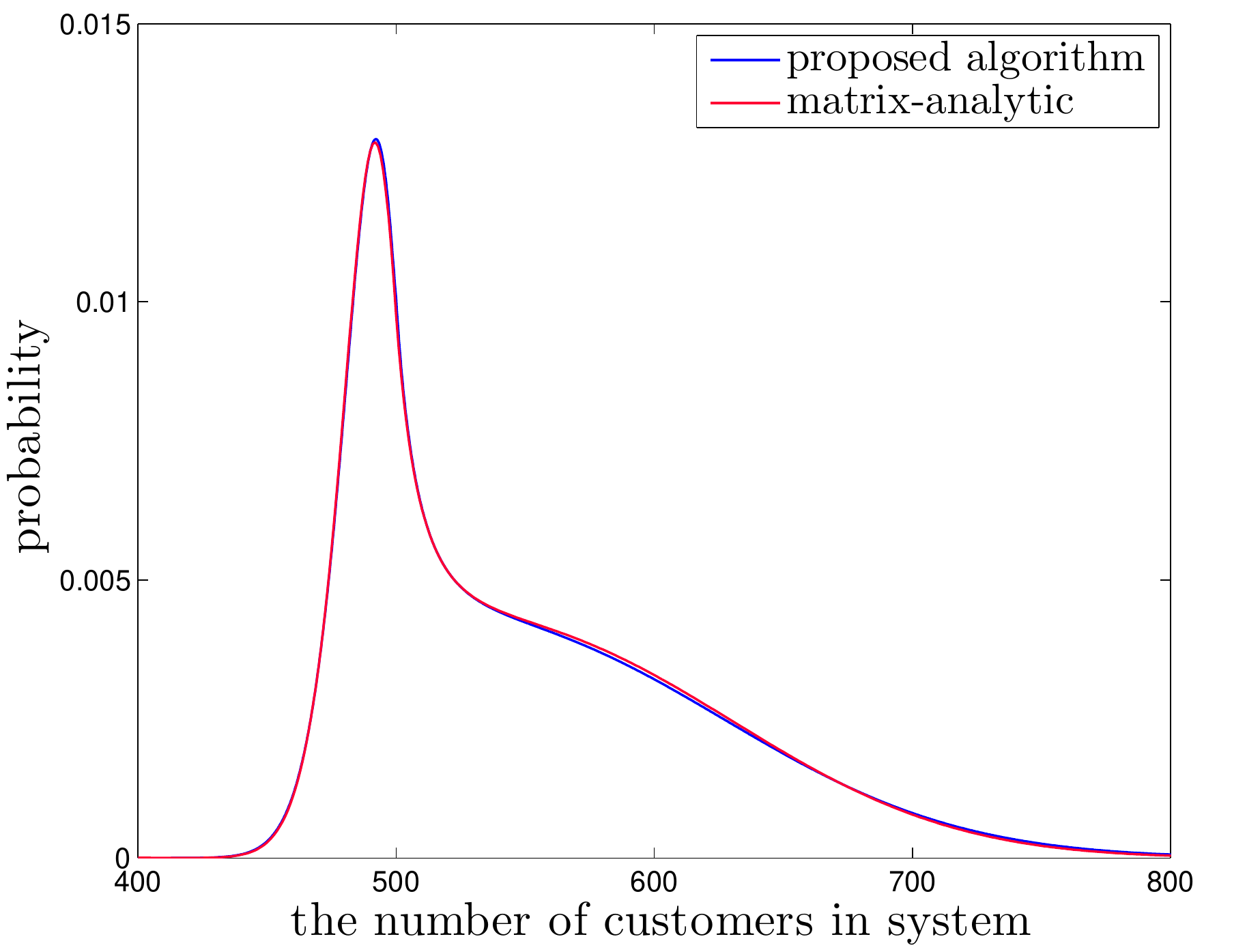}}\caption{The output of the
proposed algorithm with different meshes.}%
\label{fig:FE}%
\end{figure}

\begin{table}[t]
\centering
\begin{tabular}
[c]{l|lll}
& $0.5\times0.5$ & $0.25 \times0.25$ & Matrix-analytic\\\hline
Mean queue length & $54.17$ & $54.17$ & $54.05$\\
Abandonment fraction & $0.05181$ & $0.05182$ & $0.05173$\\
$\mathbb{P}[N(\infty)>470]$ & $0.9701$ & $0.9702$ & $0.9694$\\
$\mathbb{P}[N(\infty)>500]$ & $0.6838$ & $0.6835$ & $0.6818$\\
$\mathbb{P}[N(\infty)>600]$ & $0.2244$ & $0.2241$ & $0.2229$\\
$\mathbb{P}[N(\infty)>750]$ & $0.008233$ & $0.008246$ & $0.006395$%
\end{tabular}
\caption{The output of the proposed algorithm using different meshes.}%
\label{tab:finer_mesh}%
\end{table}

\begin{figure}[t]
\centering
\subfloat[$0.25\times 0.25$ finite elements] {\label{fig:Naive025}
\includegraphics[width=2.45in]{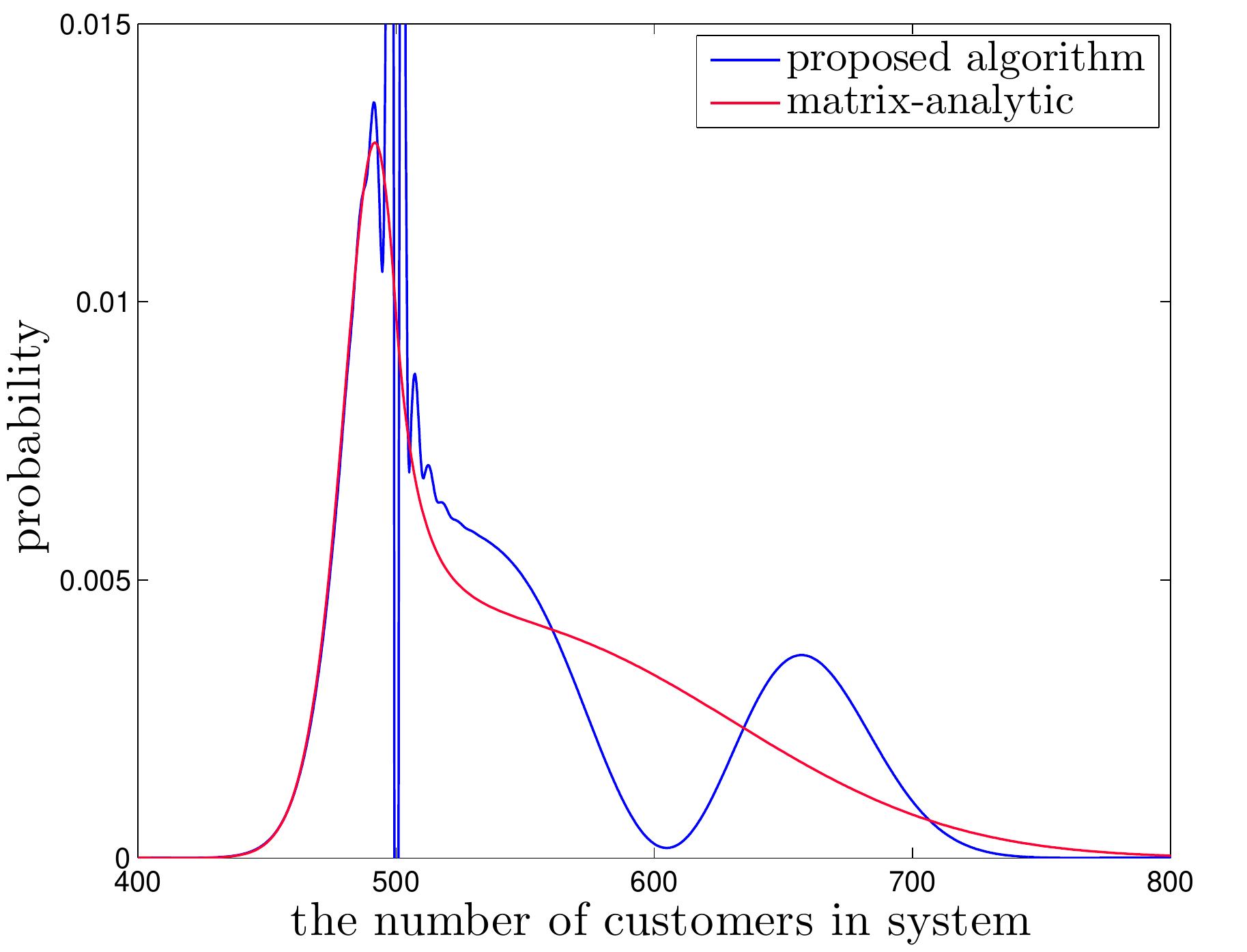}}
\subfloat[$0.125\times 0.125$ finite elements] {\label{fig:Naive01}
\includegraphics[width=2.45in]{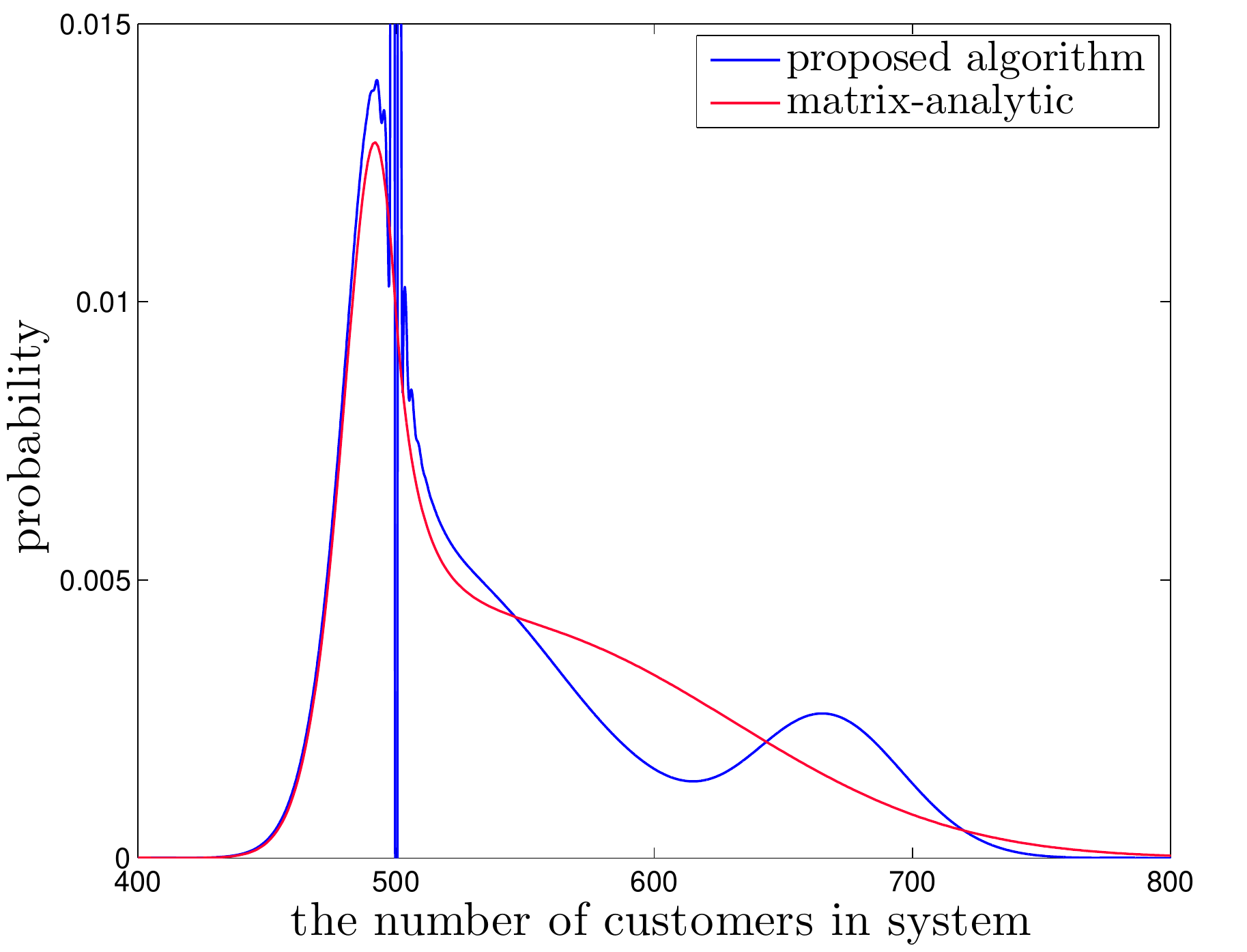}}\caption{The output
of the proposed algorithm with the \textquotedblleft naive\textquotedblright%
\ reference density and different meshes.}%
\label{fig:Naive}%
\end{figure}

When both the reference density and the truncation hypercube are fixed, using
a finer mesh may produce smaller approximation error. However, a finer mesh
yields more basis functions, which in turn lead to a larger condition number
for the matrix $A$ in (\ref{eq:equation}). If the condition number of $A$ is
too large, the round-off error in solving (\ref{eq:equation}) becomes
considerable. So a finer mesh does not necessarily yield a more accurate output.

Let us test different meshes for the second scenario in
Section~\ref{sec:Example1}. We keep the same settings for the algorithm except
the size of finite elements. The output with $1.0\times1.0$ finite elements is
plotted in Figure~\ref{fig:FE1}. With this mesh, the algorithm does not
perform well at the intervals where the stationary density varies quickly. We
need a finer mesh to improve the accuracy. In this case, the condition number
of $A$ is $5.70\times10^{20}$. Recall that to produce the curve in
Figure~\ref{fig:MH2M500}, we use a mesh consisting of $0.5\times0.5$ finite
elements. With this mesh, the condition number of $A$ is $1.15\times10^{23}$.
When the element size is further reduced to $0.25\times0.25$, the condition
number of $A$ grows to $7.13\times10^{27}$. As illustrated in
Figure~\ref{fig:FE025}, the output of the algorithm fits the exact stationary
distribution well. When we compare Figures~\ref{fig:MH2M500}
and~\ref{fig:FE025}, however, there is barely any difference noticeable
between the algorithm outputs. To confirm that this mesh is not superior to
the one with $0.5\times0.5$ finite elements, we list several performance
estimates in Table~\ref{tab:finer_mesh}. In this table, the results in
Table~\ref{tab:MH2500M} are duplicated for comparison purposes. The difference
between the algorithm outputs using these two meshes is negligible.
Considering the modeling error of the diffusion model, we can assert that
using $0.5\times0.5$ finite elements is sufficient to produce an accurate
approximation for this queue.

Given an appropriate reference density and the associated truncation
hypercube, the above discussion has demonstrated an approach to selecting a
mesh. Beginning with two meshes, with one finer than the other, we compare the
algorithm outputs using these two meshes. If obvious difference is observed,
the coarser mesh should be discarded and a further finer mesh is explored.
Continue this procedure until the difference between the outputs of two meshes
are negligible. Then, the coarser\ one of the remaining two is selected as an
appropriate mesh.

We would also demonstrate that with an improper reference density, a finer
mesh cannot make the algorithm yield an adequate output. Let us go back to the
example in Section~\ref{sec:InfluenceReference} with the \textquotedblleft
naive\textquotedblright\ reference density. We set the truncation rectangle to
be $K=[-7,32]\times\lbrack-7,32]$ and the size of finite elements to be
$0.25\times0.25$. The output is shown in Figure~\ref{fig:Naive025}. Although
the curve appears smoother than the one in Figure~\ref{fig:MH2500naive2} with
$0.5\times0.5$ finite elements, the output still fails to capture the exact
stationary distribution. This time, the condition number of $A$ is
$3.91\times10^{195}$. There is no doubt that such an ill-conditioned matrix
will bring about a huge round-off error in solving (\ref{eq:equation}). A mesh
with $0.125\times0.125$ finite elements is also investigated and the algorithm
output is plotted in Figure~\ref{fig:Naive01}. The condition number of $A$
increases to $6.35\times10^{198}$ and the algorithm misses the target as well.

\subsection{Gauss-Legendre quadrature}

\begin{table}[t]
\centering
\begin{tabular}
[c]{l|llll}
& $m=4$ & $m=8$ & $m=16$ & Matrix-analytic\\\hline
Mean queue length & $54.17$ & $54.17$ & $54.17$ & $54.05$\\
Abandonment fraction & $0.05181$ & $0.05181$ & $0.05181$ & $0.05173$\\
$\mathbb{P}[N(\infty)>470]$ & $0.9701$ & $0.9701$ & $0.9701$ & $0.9694$\\
$\mathbb{P}[N(\infty)>500]$ & $0.6833$ & $0.6838$ & $0.6839$ & $0.6818$\\
$\mathbb{P}[N(\infty)>600]$ & $0.2245$ & $0.2244$ & $0.2244$ & $0.2229$\\
$\mathbb{P}[N(\infty)>750]$ & $0.008235$ & $0.008233$ & $0.008232$ &
$0.006395$%
\end{tabular}
\caption{The output of the proposed algorithm with different quadrature
orders.}%
\label{tab:degrees}%
\end{table}

Before solving the linear system (\ref{eq:equation}), we must generate the
matrix $A$ and the vector $v$ whose entries are given by (\ref{eq:coeff}). We
follow a Gauss-Legendre quadrature rule to compute the integral for each
entry. The integral is taken over a two-dimensional rectangle and the
quadrature rule evaluates the integrand at $m$ points in each dimension. The
results are more accurate when a larger $m$ is used. In
Section~\ref{sec:Numerical}, we take $m=8$ in the numerical examples. Here, we
briefly discuss the impact of the order $m$.

Several performance estimates are listed in Table~\ref{tab:degrees}. We keep
the same settings for the algorithm\ except the quadrature order in each
dimension. For the convenience of comparison, the results in
Table~\ref{tab:MH2500M} are duplicated in this table. Clearly, the
Gauss-Legendre quadrature of order $m\geq4$ is sufficiently accurate for our purposes.

\subsection{Computational complexity}

\label{sec:complexity}

\begin{table}[t]
\centering
\begin{tabular}
[c]{l|llll}
& $1.0\times1.0$ & $0.5\times0.5$ & $0.25\times0.25$ & $0.125\times
0.125$\\\hline
Dimension $m_{C}$ & $5776$ & $23716$ & $96100$ & $386884$\\
Constructing $A$ and $v$ & $6.63$ & $27.3$ & $109$ & $455$\\
Solving (\ref{eq:equation}) & $0.0780$ & $0.359$ & $2.29$ & $18.2$%
\end{tabular}
\caption{Computation time (in seconds) of the proposed algorithm using
different meshes.}%
\label{tab:Time}%
\end{table}

Let $d$, the dimension of the diffusion model, be fixed. The size of $A$ is
$m_{C}\times m_{C}$ where $m_{C}$ is the dimension of the functional space $C$
given by (\ref{eq:Nn}). The matrix $A$\ is sparse. There are at most $6^{d}$
nonzero entries in each row or column. Hence, it takes $O(m_{C})$ arithmetic
operations to construct $A$. We may apply Gaussian elimination to solve the
linear system (\ref{eq:equation}). When the basis functions are properly
ordered, the nonzero entries of $A$ are confined to a diagonally bordered band
of width $O(m_{C}^{(d-1)/d})$. Hence, solving (\ref{eq:equation}) requires
$O(m_{C}^{(2d-1)/d})$ operations as $m_{C}\rightarrow\infty$.

The computation time (measured by seconds) for various meshes can be found in
Table~\ref{tab:Time}, where we list both the time for constructing $A$ and $v$
and the time for solving (\ref{eq:equation}). When computing $A$ and $v$, we
follow a Gauss-Legendre quadrature rule with $m=8$ points in each dimension.
The truncation rectangle is set to be $K=[-7,32]\times\lbrack-7,32]$. Each
mesh is obtained by setting the size of finite elements. The dimension $m_{C}$
increases by around four times as the width of each finite element is reduced
by half. The proposed algorithm is tested on a laptop with a 2.66GHz Intel
Core 2 Duo processor and eight gigabytes memory. Both $A$ and $v$ are produced
by our \textrm{C\raise.3ex\hbox{\small++}} package. The linear system
(\ref{eq:equation}) is solved by Matlab. These two parts are connected via a
MEX interface that comes with Matlab.

\section{Concluding remarks}

\label{sec:Conclusion}

In this paper, we proposed two approximate models for many-server queues with
customer abandonment. Both these models are diffusion processes and they
differ in how the abandonment process is approximated. A finite element
algorithm was proposed for computing the stationary distribution of each
model. The essential part of the algorithm is a reference density that
controls the convergence of the algorithm. To construct a reference density,
we conjectured that the limit queue length process has a certain Gaussian
tail. Using this conjecture, we proposed a systematic approach to choosing a
reference density. With the proposed reference density, the output of the
algorithm is stable and accurate. Numerical examples indicate that the
diffusion models are good approximations for many-server queues.

Assume that the stationary density $g$ is twice differentiable in
$\mathbb{R}^{d}$ and vanishes at infinity. Using the basic adjoint
relationship (\ref{eq:bar_density}) and applying integration by parts twice,
we have
\[
\mathcal{G}^{\ast}g(x)=0\qquad\text{for all }x\in\mathbb{R}^{d}%
\]
where $\mathcal{G}^{\ast}$ is the adjoint operator of the generator
$\mathcal{G}$. Fix a finite domain $K\subset\mathbb{R}^{d}$ large enough. One
can solve the stationary density $g$ by the Dirichlet problem%
\[%
\begin{cases}
\mathcal{G}^{\ast}g(x)=0 & \text{for }x\text{ in the interior of }K\text{,}\\
g(x)=0 & \text{for }x\text{ on the boundary of }K\text{.}%
\end{cases}
\]
Such a Dirichlet problem can be solved via a finite difference algorithm.
Alternatively, for each test function $f$, one may apply integration by parts
once to the basic adjoint relationship to obtain an equation that involves the
first derivatives of $g$ and the first derivatives of $f$. From this weak
formulation, fixing a large enough finite domain $K$ and assuming that $g$ is
zero on the boundary of $K$, one may apply a standard Galerkin finite element
method to compute the stationary density $g$ on $K$. See, e.g.,
\cite{KovalovLinetskyMarcozzi07}. Both the finite difference algorithm and the
Galerkin method do not use a reference density. A future research topic is to
compare the efficiency and accuracy of these two algorithms with the proposed
algorithm in this paper.

The dimension of the functional space $C$ in Section~\ref{sec:FEM} grows
exponentially in $d$, the dimension of the diffusion model. As a consequence,
both the computation time and the memory usage increases exponentially in $d$.
When $d$ is not small, the curse of dimensionality is a serious challenge for
the proposed algorithm as well as any other algorithms. To reduce the
dimension of $C$, one possible approach is to investigate a reference density
that potentially shares more common features with the stationary density. Such
a reference density may enable us to compute the stationary density with a
moderate number of basis functions when $d$ is not small. Another possible
direction to reduce the computational complexity of the algorithm is to
investigate a low-rank matrix approximation for the linear system
(\ref{eq:equation}). The technique of random sampling may be explored. See
\cite{KannanVempala10} for more details.

\appendix

\section*{Appendix}

\subsection*{Proof of Proposition \ref{prop:independence}}

Recall that $K$ is the compact support of $C$ and the basis functions of $C$
are given by (\ref{eq:basisFunctions}). We use $C_{0}^{1}(K)$ to denote the
set of real-valued functions on a neighborhood of $K$ that are continuously
differentiable and have compact support in $K$. Clearly, $C\subset C_{0}%
^{1}(K)$. For any $f,\hat{f}\in C_{0}^{1}(K)$, we define an inner product by
\[
\langle f,\hat{f}\rangle_{D(K)}=\sum_{j=1}^{d}\int_{K}\frac{\partial
f(x)}{\partial x_{j}}\frac{\partial\hat{f}(x)}{\partial x_{j}}\,\mathrm{d}x
\]
and let $W_{0}^{1,2}(K)$ be the closure of $C_{0}^{1}(K)$ in the norm induced
by this inner product. Then, $W_{0}^{1,2}(K)$ is a Hilbert space and $C\subset
W_{0}^{1,2}(K)$.

\begin{proof}[Proof of Proposition \ref{prop:independence}]
Since $\mathcal{G}$ is a linear operator, it suffices to show that for any
$f_{0}\in C$, we must have $f_{0}=0$ if $\mathcal{G}f_{0}=0$ in $L^{2}(\mathbb{R}^{d},r)$.
The uniform elliptic operator $\mathcal{G}$ can be written into the divergence
form as in (8.1) of \cite{GilbargTrudinger01}, i.e.,
\[
\mathcal{G}f(x)=\sum_{j=1}^{d}\hat{b}_{j}(x)\frac{\partial f(x)}{\partial
x_{j}}+\frac{1}{2}\sum_{j=1}^{d}\sum_{\ell=1}^{d}\frac{\partial(\Sigma_{j\ell
}(x)\partial f(x)/\partial x_{j})}{\partial x_{\ell}}
\]
for each $f\in C_{b}^{2}(\mathbb{R}^{d})$, where\[
\hat{b}_{j}(x)=b_{j}(x)-\frac{1}{2}\sum_{\ell=1}^{d}\frac{\partial
\Sigma_{j\ell}(x)}{\partial x_{\ell}}.
\]
Let $U\subset\mathbb{R}^{d}$ be a connected open set that is bounded and
contains $K$. Since $r>0$ and $\mathcal{G}f_{0}$ is continuous in the interior
of each finite element, we must have $\mathcal{G}f_{0}=0$ in $K$ except on the
boundaries of certain finite elements where $\mathcal{G}f_{0}$ is not defined.
Hence, $\mathcal{G}f_{0}=0$ in $U$ in the weak sense (see (8.2) of
\cite{GilbargTrudinger01}). Note that $b$, $\Sigma$, and the partial
derivatives of $\Sigma$ are all continuous, so both $\hat{b}$ and $\Sigma$ are
bounded in $U$. Because $f_{0}\in W_{0}^{1,2}(K)$, it follows from
Corollary~8.2 of \cite{GilbargTrudinger01} that $f_{0}=0$ in $K$, and thus
$f_{0}=0$ in $\mathbb{R}^{d}$.
\end{proof}

\subsection*{Proof of Proposition \ref{prop:Hn}}

Given a compact set $K\subset\mathbb{R}^{d}$, let $C_{b}^{2}(K)$ be the set of
real-valued functions on a neighborhood of $K$ that are twice continuously
differentiable with bounded first and second derivatives in $K$. For each
$f\in C_{b}^{2}(K)$, define a norm $\left\Vert \cdot\right\Vert _{H^{2}(K)}$
by%
\[
\left\Vert f\right\Vert _{H^{2}(K)}^{2}=\int_{K}\bigg(f^{2}(x)+\max
_{j=1,\ldots,d}\Big(\frac{\partial f(x)}{\partial x_{j}}\Big)^{2}+\max
_{j,\ell=1,\ldots,d}\Big(\frac{\partial^{2}f(x)}{\partial x_{j}\partial
x_{\ell}}\Big)^{2}\bigg)r(x)\,\mathrm{d}x.
\]
Because both $b$ and $\Sigma$ are bounded in $K$, there exists $\kappa
_{0}(K)>0$ such that
\begin{equation}
\int_{K}(\mathcal{G}f(x))^{2}r(x)\,\mathrm{d}x\leq\kappa_{0}(K)\left\Vert
f\right\Vert _{H^{2}(K)}^{2}\qquad\text{for all }f\in C_{b}^{2}(K)\text{.}%
\label{eq:H2K}%
\end{equation}
Let $\bar{C}_{b}^{2}(K)$ be the closure of $C_{b}^{2}(K)$ in the above norm. A
standard procedure can be used to define the first-order and the second-order
derivatives for each $f\in\bar{C}_{b}^{2}(K)$. Then, the operator
$\mathcal{G}$ can be extended to $\bar{C}_{b}^{2}(K)$ and inequality
(\ref{eq:H2K}) holds for all $f\in\bar{C}_{b}^{2}(K)$.

\begin{proof}[Proof of Proposition \ref{prop:Hn}]
It suffices to prove that for any $f_{0}\in C_{b}^{2}(\mathbb{R}^{d})$, there
exists a sequence of functions $\{\varphi_{k}\in C_{k}:k\in\mathbb{N}\}$ such
that
\[
\left\Vert \mathcal{G}\varphi_{k}-\mathcal{G}f_{0}\right\Vert
\rightarrow0 \qquad\text{as } k\rightarrow\infty.
\]
Fix $\varepsilon>0$. Because $K_{k}\uparrow\mathbb{R}^{d}$ as $k\rightarrow
\infty$, by (\ref{eq:Gassump}) and the Cauchy-Schwartz inequality, there
exists $a\in\mathbb{N}$ such that%
\begin{equation}
\int_{\mathbb{R}^{d}\setminus K_{a}}(\mathcal{G}f_{0}(x))^{2}r(x)\,\mathrm{d}%
x<\frac{\varepsilon^{2}}{2}.\label{eq:RK}%
\end{equation}
Consider the finite hypercube $K_{a}$. By (\ref{eq:H2K}), there exists
$\kappa_{0}(K_{a})>0$ such that
\begin{equation}
\int_{K_{a}}(\mathcal{G}f(x))^{2}r(x)\,\mathrm{d}x\leq\kappa_{0}%
(K_{a})\left\Vert f\right\Vert _{H^{2}(K_{a})}^{2}\qquad\text{for all }%
f\in\bar{C}_{b}^{2}(K_{a})\text{.} \label{eq:Ka}%
\end{equation}
A polynomial can be used to approximate $f_{0}$ on $K_{a}$. By Proposition 7.1
in the appendix of \cite{ethkur86}, there exists a polynomial
$f_{\operatorname*{p}}$ such that
\[
\left\Vert f_{\operatorname*{p}}-f_{0}\right\Vert _{H^{2}(K_{a})}%
<\frac{\varepsilon}{2\sqrt{2\kappa_{0}(K_{a})}}.
\]
For the lattice mesh $\Delta_{k}$, let $\Lambda_{a,k}$ be the set of its nodes
in the interior of $K_{a}$. For any $k\geq a$, let $\varphi_{k}$ be a function
in $C_{k}$ such that $\varphi_{k}(x)=0$ for all $x\in\mathbb{R}^{d}\setminus
K_{a}$ and%
\[
\varphi_{k}(x)=f_{\operatorname*{p}}(x)\qquad\text{and}\qquad\frac
{\partial\varphi_{k}(x)}{\partial x_{j}}=\frac{\partial f_{\operatorname*{p}%
}(x)}{\partial x_{j}}%
\]
for $j=1,\ldots,d$ and all $x\in
\Lambda_{a,k}$.
Clearly, $\varphi_{k}\in\bar{C}_{b}^{2}(K_{a})$. Because the sequence of
meshes $\{\Delta_{k}:k\in\mathbb{N}\}$ is regularly refined, there exists a
constant $\kappa_{1}>0$ such that $\eta_{\Delta_{k}}<\kappa_{1}$ for all
$k\geq a$. Using the interpolation error estimate in Theorem 6.6 of
\cite{OdenReddy76}, we have%
\[
\left\Vert \varphi_{k}-f_{\operatorname*{p}}\right\Vert _{H^{2}(K_{a})}%
\leq\kappa_{1}^{2}\kappa_{2}\kappa_{3}\Big(\int_{\mathbb{R}^{d}}%
r(x)\,\mathrm{d}x\Big)^{1/2}\left\vert \Delta_{k}\right\vert ^{2},
\]
where $\kappa_{2}>0$ is a constant independent of $\Delta_{k}$ and
$f_{\operatorname*{p}}$, and%
\[
\kappa_{3}=\sup\left\{  \left\vert \frac{\partial^{4}f_{\operatorname*{p}}%
(x)}{\partial x_{1}^{m_{1}}\cdots\partial x_{d}^{m_{d}}}\right\vert :x\in
K_{a};\,m_{1}+\cdots+m_{d}=4\right\}  <\infty.
\]
Hence, there exists $\delta_{0}>0$ such that
\[
\left\Vert \varphi_{k}-f_{\operatorname*{p}}\right\Vert _{H^{2}(K_{a})}%
<\frac{\varepsilon}{2\sqrt{2\kappa_{0}(K_{a})}}%
\]
whenever $\left\vert \Delta_{k}\right\vert <\delta_{0}$. In this case,
\[
\left\Vert \varphi_{k}-f_{0}\right\Vert _{H^{2}(K_{a})}\leq\left\Vert
\varphi_{k}-f_{\operatorname*{p}}\right\Vert _{H^{2}(K_{a})}+\left\Vert
f_{\operatorname*{p}}-f_{0}\right\Vert _{H^{2}(K_{a})}<\frac{\varepsilon
}{\sqrt{2\kappa_{0}(K_{a})}}.
\]
By (\ref{eq:Ka}),%
\begin{equation}
\int_{K_{a}}(\mathcal{G}\varphi_{k}(x)-\mathcal{G}f_{0}(x))^{2}%
r(x)\,\mathrm{d}x\leq\kappa_{0}(K_{a})\left\Vert \varphi_{k}-f_{0}\right\Vert
_{H^{2}(K_{a})}^{2}<\frac{\varepsilon^{2}}{2}.\label{eq:GfK}%
\end{equation}
It follows from (\ref{eq:RK}) and (\ref{eq:GfK}) that
$\left\Vert \mathcal{G}\varphi_{k}-\mathcal{G}f_{0}\right\Vert <\varepsilon$
whenever $k\geq a$ and $\left\vert \Delta_{k}\right\vert <\delta_{0}$.
\end{proof}

\section*{Acknowledgements}

The authors would like to thank Ton Dieker and Vadim Linetsky for their
helpful comments.
\bibliographystyle{imsart-nameyear}
\bibliography{dai}

\end{document}